\documentclass[11pt,reqno]{amsart}

\usepackage[margin=1.15in]{geometry}
\usepackage{amsmath,amssymb,amsthm,mathtools}
\usepackage{enumitem}
\usepackage[colorlinks=true,linkcolor=blue!55!black,citecolor=blue!55!black,urlcolor=blue!55!black]{hyperref}
\usepackage{tikz}

\theoremstyle{plain}
\newtheorem{theorem}{Theorem}[section]

\newtheorem{lemma}[theorem]{Lemma}
\newtheorem{corollary}[theorem]{Corollary}
\theoremstyle{definition}
\newtheorem{definition}[theorem]{Definition}
\newtheorem{example}[theorem]{Example}
\newtheorem{remark}[theorem]{Remark}
\newtheorem{hypothesis}[theorem]{Hypothesis}
\newtheorem{question}[theorem]{Question}

\newcommand{\Sgm}{\operatorname{Sgm}}
\newcommand{\Oz}{\mathcal{O}_{\mathbb{Z}}}
\newcommand{\bb}{\mathbf{b}}
\newcommand{\ee}{\mathbf{e}}
\newcommand{\supp}{\operatorname{supp}}
\newcommand{\Z}{\mathbb{Z}}
\newcommand{\Zpos}{\Z_{\geq 0}}
\newcommand{\dv}{{}^{\vee}}
\DeclareMathOperator{\pr}{pr}

\begin{document}
\title[Duality, rigidity, and peeling for multisegments]
{Duality, rigidity, and peeling for multisegments:\\ on hypotheses of Mitra, Offen, and Sayag}
\author{Hariom Sharma}
\email{hariomshrma97@gmail.com, hariom@math.iitb.ac.in}
\address{Department of Mathematics, Indian Institute of Technology, Bombay, Powai, Mumbai, Maharashtra 400076, India}

\begin{abstract}
    Mitra, Offen, and Sayag introduced the notions of distinguished multisegments and multisegments of Speh type, and formulated the hypothesis that every distinguished multisegment is of Speh type \cite[Hypothesis~8.5]{MOS}. They also proposed the related hypothesis that a multisegment which is distinguished together with its dual is of Speh type \cite[Hypothesis~8.6]{MOS}. They proved these statements for sets of segments and for multisegments in which at most two segments share a given endpoint.
In this article, we develop a systematic combinatorial theory of relevant decompositions and establish several structural properties of distinguished multisegments. We prove that relevance is preserved under the involution $\Delta\mapsto\Delta\dv$ together with reversal of the standard order. As a consequence, a multisegment $\mathfrak m$ is distinguished if and only if its dual $\mathfrak m\dv$ is distinguished, and Hypotheses~8.5 and~8.6 of \cite{MOS} are equivalent. We introduce the numerical invariant
$S_{\mathfrak m}(\Delta) = \sum_{t\geq0}(-1)^t\mathfrak m(\nu^t\Delta)$,
and prove that $\mathfrak m$ is of Speh type if and only if
$S_{\mathfrak m}(\Delta)\geq0$ for every segment $\Delta$.
Using the rigidity properties of relevant decompositions, we obtain a peeling theorem that controls the extremal layers of a distinguished multisegment. In particular, if $\mathfrak m$ is distinguished, then the top endpoint layer is supported by the preceding layer, namely $\nu^{-1}\mathfrak m_{c}\leq \mathfrak m_{c-1}$,
where $c$ is the largest endpoint occurring in $\mathfrak m$, together with the corresponding dual inequality at the smallest beginning. These constraints lead to a numerical tameness condition, and we prove both hypotheses of Mitra--Offen--Sayag for every tame multisegment. The class of multisegments for which $\mathfrak m$ or its dual
$\mathfrak m^\vee$ is tame strictly extends the classes previously
treated by Mitra--Offen--Sayag.
Finally, we construct a five-segment multisegment which is not distinguished, but for which every standard order with non-increasing endpoints, including the canonical order of \cite[\S8.0.12]{MOS}, admits a non-trivial relevant decomposition. Thus, the witness strategy based only on such orders cannot establish Hypothesis~8.5 in full generality. 
The general case remains open, but our results reduce it to the class of
multisegments $\mathfrak m$ for which both $\mathfrak m$ and its dual
$\mathfrak m^\vee$ are non-tame. In addition, we establish structural
constraints on the standard orders that may serve as witnesses for
non-distinction.
\end{abstract}
\maketitle
\section{Introduction}
Let $\mathrm{F}$ be a $p$-adic field. A central problem in the relative Langlands program is to determine which irreducible representations of a reductive group admit a period with respect to a prescribed subgroup. Such periods detect additional symmetries of representations and often reflect distinguished features of their Langlands parameters. For the Klyachko subgroups of $\mathrm{GL}_n(\mathrm{F})$, these questions were studied
extensively by Offen--Sayag \cite{OS2} and, in the case of ladder
representations, by Mitra, Offen, and Sayag \cite{MOS}.

Using the geometric lemma of Bernstein--Zelevinsky \cite{BZ} to analyze the standard modules $\lambda(\mathfrak{m})$ associated with multisets $\mathfrak{m}$ of Zelevinsky segments \cite{Z}, the authors reduced a necessary condition for symplectic distinction of standard modules to a purely combinatorial problem, formulated independently in \cite[Section $8$]{MOS}. Given a standard order on a multisegment of integer segments, they introduced the notion of a \emph{relevant decomposition}. A multisegment for which every standard order admits a relevant decomposition is called \emph{distinguished}. They also defined a multiset $\mathfrak{m}$ to be \emph{of Speh type} if
$$\mathfrak{m}=\mathfrak{n}+\nu\mathfrak{n}$$
for some multiset $\mathfrak{n}$, where 
$$\nu([a,b])=[a+1,b+1].$$
Every multisegment of Speh type is distinguished. Mitra, Offen, and Sayag proposed that the converse should also hold.
\begin{hypothesis}[{\cite[Hypothesis~8.5]{MOS}}]\label{hyp:85}
If $\mathfrak m$ is distinguished, then $\mathfrak m$ is of Speh type.
\end{hypothesis}

They also formulated the following apparently weaker statement.

\begin{hypothesis}[{\cite[Hypothesis~8.6]{MOS}}]\label{hyp:86}
If both $\mathfrak m$ and $\mathfrak m\dv$ are distinguished, then
$\mathfrak m$ is of Speh type.
\end{hypothesis}

Here, $\mathfrak m\dv$ is obtained by reflecting each segment:
$$[a,b]\dv=[-b,-a].$$
Since the Speh type condition is preserved by this duality, a multisegment of Speh type has both $\mathfrak m$ and $\mathfrak m\dv$ distinguished. Thus, Hypothesis~\ref{hyp:85}, if valid in general, would give the precise combinatorial characterization
$$\mathfrak m\text{ is distinguished}
\quad\Longleftrightarrow\quad
\mathfrak m\text{ is of Speh type},$$
whereas Hypothesis~\ref{hyp:86} would characterize the simultaneous distinguishedness of $\mathfrak m$ and $\mathfrak m\dv$. We discuss the different converse statements involved in \S\ref{ss:converses}, including an order-level converse that turns out to be false.
The authors of \cite{MOS} deliberately state these as hypotheses rather than conjectures and prove Hypothesis~\ref{hyp:85} in two important cases: when $\mathfrak{m}$ is a set and when no endpoint occurs in more than two segments. In both cases, they show that a canonical standard order admits no non-trivial relevant decomposition, following the strategy proposed in \cite[\S8.0.10]{MOS}. 

The purpose of this article is to advance the combinatorial theory behind these hypotheses.
Our contributions are threefold. First, we establish structural results, including duality, numerical, and rigidity properties, that apply to arbitrary multisegments. Second, we prove both hypotheses for a new class of multisegments that strictly extends the cases previously treated by Mitra--Offen--Sayag. Third, we exhibit an obstruction showing that the strategy based exclusively on standard orders with non-increasing endpoints cannot establish Hypothesis~\ref{hyp:85} in full generality. Throughout, $\mathfrak{m}[1],\mathfrak m[2],\dots$ denotes the decomposition of $\mathfrak{m}$ according to the
(strictly decreasing) endpoint values occurring in it, as in \cite[\S8.0.11]{MOS}; we also write
$\mathfrak{m}_{x}$ for the sub-multiset of segments with endpoint $x$ and $\mathfrak{m}^{x}$ for the
sub-multiset of segments with beginning $x$. 
We refer the reader to Section~\ref{sec:setting} for notation and terminology used below without further definition.

Our first main theorem removes the distinction between the two hypotheses.

\begin{theorem}[Duality; Theorem~\ref{thm:duality} and Corollary~\ref{cor:85iff86}]\label{thm:A}
Let $\mathfrak{m}\in\Oz$.  A standard order of $\mathfrak{m}$ admits a relevant decomposition if and only if the
reversed dual order of $\mathfrak{m}\dv$ does.  Consequently, $\mathfrak{m}$ is distinguished if and only if
$\mathfrak{m}\dv$ is distinguished, and Hypotheses~\ref{hyp:85} and~\ref{hyp:86} are equivalent.
\end{theorem}

Thus, the additional condition in Hypothesis~\ref{hyp:86}, namely that
$\mathfrak{m}^{\vee}$ is distinguished, is automatically satisfied whenever
$\mathfrak{m}$ is distinguished. Consequently, Hypothesis~\ref{hyp:86}
holds unconditionally for all sets of segments, as well as for all
multisegments satisfying $\lvert \mathfrak{m}[i]\rvert\leq 2$ for every $i$.
By duality, Hypothesis~\ref{hyp:85} holds whenever at most two segments of
$\mathfrak{m}$ have the same beginning; see Corollary~\ref{cor:dual88}.
Thus, the duality theorem transfers results involving multiplicities of
endpoints to corresponding results involving multiplicities of beginnings.

Our second main result provides a numerical criterion for a multisegment 
to be of Speh type and serves as the basis for the subsequent analysis.

\begin{theorem}[Theorem~\ref{thm:speh}]\label{thm:B}
For each $\Delta\in\Sgm$, define
$S_{\mathfrak{m}}(\Delta)
    =
    \sum_{t\geq 0}(-1)^t\,
    \mathfrak{m}\bigl(\nu^t\Delta\bigr)$.
Then $\mathfrak{m}$ is of Speh type if and only if $ S_{\mathfrak{m}}(\Delta)\geq 0$
for every segment $\Delta$.
\end{theorem}

The integers $S_{\mathfrak{m}}(\Delta)$ are precisely the multiplicities 
obtained by successively peeling $\mathfrak{m}$ from the top. More 
explicitly, one must first match the segments having the largest endpoint 
with their $\nu^{-1}$-translates and then repeat this procedure with the 
remaining segments.

Our third main result shows that a distinguished multisegment necessarily 
survives the first step of the peeling procedure. It also introduces a 
mechanism for carrying out the peeling within a single standard order, 
namely, a family of standard orders that we call \emph{admissible}.

\begin{theorem}[Peeling; Theorem~\ref{thm:peel} and 
Corollary~\ref{cor:firstlevel}]
\label{thm:C}
Let $\mathfrak{m}\neq 0$ be a distinguished multisegment, and let $c$ be 
the largest endpoint occurring among the segments of $\mathfrak{m}$. Then
$$\nu^{-1}\mathfrak{m}_{c}\leq \mathfrak{m}_{c-1}.$$
Equivalently, $S_{\mathfrak{m}}(\Delta)\geq 0$
for every segment $\Delta$ satisfying $\ee(\Delta)=c-1$. Dually, if $a$ 
is the smallest beginning occurring among the segments of $\mathfrak{m}$, 
then
$$ \nu\mathfrak{m}^{a}\leq \mathfrak{m}^{a+1}.$$
\end{theorem}
The obstruction to iterating the peeling procedure arises from a genuine 
tension between two ordering requirements imposed on 
$\mathfrak{m}_{c-1}$; see Remark~\ref{rem:tension}. This obstruction disappears under the following condition.
\begin{definition}\label{def:tameintro}
Let $\mathfrak{m}$ be a multisegment. An integer $x\in\Z$ is said to be \emph{stable for $\mathfrak{m}$} if $S_{\mathfrak{m}}(\Gamma)\geq 0$
for every segment $\Gamma$ satisfying $\ee(\Gamma)\geq x$.
We say that $\mathfrak{m}$ is \emph{tame} if the following condition holds:
for every integer $x$ that is stable for $\mathfrak{m}$ and every segment $\Delta$ satisfying $\ee(\Delta)=x$, $S_{\mathfrak{m}}(\Delta)>0$, and $S_{\mathfrak{m}}(\nu\Delta)>0$, 
one has $\bb(\Delta)\geq \bb(\Gamma)$
for every segment $\Gamma$ such that $\ee(\Gamma)=x$ and $S_{\mathfrak{m}}(\Gamma)>0$. 
\end{definition}
\begin{theorem}[Theorem~\ref{thm:main}]\label{thm:D}
Hypotheses~\ref{hyp:85} and~\ref{hyp:86} hold for every tame
multisegment, as well as for every multisegment whose dual is tame.
Moreover, every set of segments is tame, and every multisegment
$\mathfrak{m}$ satisfying $\lvert\mathfrak{m}[i]\rvert\leq 2$ \text{for all} $i$ is tame. Both inclusions are strict.
\end{theorem}

Theorem~\ref{thm:D}, together with Theorem~\ref{thm:A}, therefore
incorporates \cite[Propositions~8.7 and~8.8]{MOS} and their dual
statements into a single result. Moreover, it applies to a strictly
larger class of multisegments. 
A genuinely new example covered by Theorem~\ref{thm:main} is provided by the multisegment $$\mathfrak m={[0,1],[1,2],[1,2],[1,2],[2,3]}.$$ This multisegment is neither a set nor does it satisfy the multiplicity condition of~[1, Proposition~8.8] or its dual counterpart: indeed, $|\mathfrak m_2|=|\mathfrak m^{1}|=3$. Moreover, both extremal inequalities of Corollary~\ref{cor:firstlevel} are satisfied, so its failure of distinction is not detected by the first-order peeling constraints. Nevertheless, $\mathfrak m$ is tame. In fact, the tameness condition is genuinely non-vacuous at the stable level $x=2$ because $S_{\mathfrak m}([1,2])=2$ and $S_{\mathfrak m}([2,3])=1$. On the other hand, $$S_{\mathfrak m}([0,1])=1-3+1=-1<0,$$ so $\mathfrak m$ is not of Speh type by Theorem \ref{thm:speh}. Therefore, Theorem~\ref{thm:main} implies that $\mathfrak m$ is not distinguished.


Finally, we show that the strategy employed in \cite[\S8.0.10]{MOS}
cannot succeed in general. This strategy consists of selecting, among
the standard orders with non-increasing endpoints, an order that admits
no non-trivial relevant decomposition.

\begin{theorem}[Theorem~\ref{thm:counter}]\label{thm:E}
Let
$\mathfrak{m}_{0}
    =
    \bigl\{[0,0],[0,1],[0,1],[1,1],[1,2]\bigr\}$.
Then $\mathfrak{m}_{0}$ is neither of Speh type nor distinguished.
Nevertheless, every standard order of $\mathfrak{m}_{0}$ whose sequence
of endpoints is non-increasing admits a relevant decomposition. In
particular, the canonical order introduced in \cite[\S8.0.12]{MOS}
admits a non-trivial relevant decomposition. There are exactly three
standard orders of $\mathfrak{m}_{0}$ that admit no relevant
decomposition, two of which have non-increasing beginnings.
\end{theorem}

Theorem~\ref{thm:E} shows that standard orders with non-increasing
endpoints cannot always serve as witnesses: although
$\mathfrak{m}_{0}$ is not of Speh type, every such order admits a
non-trivial relevant decomposition. Consequently, the witness
strategy of \cite[\S8.0.10]{MOS} cannot be extended to arbitrary
multisets using only orders with non-increasing endpoints. The
duality in Theorem~\ref{thm:A} provides a natural source of other
standard orders, namely those with non-increasing beginnings.

\subsection*{Organization}
The article is organized as follows. In Section~\ref{sec:setting}, we fix notation and recall the combinatorial framework introduced in \cite[Section $8$]{MOS}. In Section~\ref{sec:speh}, we establish the numerical criterion for multisets of Speh type. Section~\ref{sec:rigid} develops the rigidity results for the first and last blocks of a relevant decomposition. In Section~\ref{sec:duality}, we prove the duality theorem and show the equivalence of the two hypotheses. Section~\ref{sec:peel} is devoted to the peeling theorem, while Section~\ref{sec:tame} contains the proof of the main theorem for tame multisets. In Section~\ref{sec:counter}, we present a counterexample showing the limitations of the strategy proposed in \cite[\S8.0.10]{MOS}. 
Section~\ref{sec:app} discusses the representation-theoretic significance of our results and several remaining open problems.

\section{The combinatorics of multisegments}\label{sec:setting}

In this section, we recall the definitions introduced in \cite[Section 8]{MOS}, using slightly expanded notation.

\subsection{multisets}

A \emph{multiset} of elements of a set $X$ is a finitely supported
function
$$\mathfrak{f}\colon X\longrightarrow \Zpos.$$
We denote its support by $\supp(\mathfrak{f})$ and define its cardinality
by
$$ \lvert\mathfrak{f}\rvert =
    \sum_{x\in X}\mathfrak{f}(x).$$
We write $x\in\mathfrak{f}$ to mean that
$x\in\supp(\mathfrak{f})$. We also use the notation
$$\mathfrak{f}=\{x_1,\ldots,x_t\},
    \qquad t=\lvert\mathfrak{f}\rvert,$$
where each element $x\in X$ occurs exactly $\mathfrak{f}(x)$ times. This listing is called an \emph{order} on $\mathfrak{f}$.

Sums and differences of multisets, as well as the relation $\leq$, are
understood pointwise. Thus,
$$\mathfrak{f}\leq\mathfrak{g}
    \quad\Longleftrightarrow\quad
    \mathfrak{f}(x)\leq\mathfrak{g}(x)
    \quad\text{for every }x\in X.$$
If $\mathfrak{f}$ takes values only in $\{0,1\}$, we identify
$\mathfrak{f}$ with its support and refer to it simply as a \emph{set}.

\subsection{Segments}

A \emph{segment} is a set of consecutive integers of the form
\[
    [a,b]=\{a,a+1,\ldots,b\},
    \qquad a\leq b.
\]
For a segment $\Delta=[a,b]$, we denote its beginning and endpoint by
\[
    \bb(\Delta)=a
    \qquad\text{and}\qquad
    \ee(\Delta)=b,
\]
respectively. We write $\Sgm$ for the set of all segments.

Define the shift and duality operations on $\Sgm$ by
\[
    \nu[a,b]=[a+1,b+1],
    \qquad
    [a,b]^\vee=[-b,-a].
\]
Both $\nu$ and $\vee$ define bijections of $\Sgm$. Moreover,
\begin{equation}\label{eq:nudual}
    (\nu\Delta)^\vee
    =
    \nu^{-1}(\Delta^\vee),
    \qquad
    \Delta^{\vee\vee}=\Delta.
\end{equation}

Let $\Delta=[a,b]$ and $\Delta'=[a',b']$ be two segments. We say that
$\Delta$ \emph{precedes} $\Delta'$, and write $\Delta\prec\Delta'$, if
\[
    a<a',
    \qquad
    b<b',
    \qquad\text{and}\qquad
    b\geq a'-1.
\]

A \emph{decomposition} of a segment $\Delta\in\Sgm$ is a tuple
\[
    (\Delta_1,\ldots,\Delta_k)\in\Sgm^k,
    \qquad k\geq 1,
\]
satisfying
\[
    \ee(\Delta_1)=\ee(\Delta),
    \qquad
    \bb(\Delta_k)=\bb(\Delta),
\]
and
\[
    \ee(\Delta_{i+1})
    =
    \bb(\Delta_i)-1,
    \qquad 1\leq i\leq k-1.
\]
Thus, a decomposition partitions $\Delta$ into $k$ non-empty consecutive
subsegments, listed in decreasing order of their endpoints. A
decomposition is called \emph{trivial} if $k=1$. Consequently, a segment
of cardinality $\ell$ has exactly $2^{\ell-1}$ decompositions.

\subsection{multisets of segments}

Let $\Oz$ denote the set of all multisets of segments. For
$\mathfrak{m}\in\Oz$ and $n\in\Z$, define
\[
    (\nu^n\mathfrak{m})(\Delta)
    =
    \mathfrak{m}(\nu^{-n}\Delta)
    \qquad\text{and}\qquad
    \mathfrak{m}^\vee(\Delta)
    =
    \mathfrak{m}(\Delta^\vee).
\]
Thus, $\nu^n\mathfrak{m}$ is obtained by shifting every segment of
$\mathfrak{m}$ by $n$, whereas $\mathfrak{m}^\vee$ is obtained by
replacing every segment by its dual.

For each $x\in\Z$, define the sub-multisets
\[
    \mathfrak{m}_x(\Delta)
    =
    \begin{cases}
        \mathfrak{m}(\Delta), & \text{if }\ee(\Delta)=x,\\
        0,                    & \text{otherwise},
    \end{cases}
    \qquad
    \mathfrak{m}^x(\Delta)
    =
    \begin{cases}
        \mathfrak{m}(\Delta), & \text{if }\bb(\Delta)=x,\\
        0,                    & \text{otherwise}.
    \end{cases}
\]
Suppose that $c_1>\cdots>c_s$
are the distinct endpoints occurring among the segments of $\mathfrak{m}$. Then, in the notation of \cite[\S8.0.11]{MOS},
$\mathfrak{m}[i]=\mathfrak{m}_{c_i}$,
and
$\mathfrak{m} = \mathfrak{m}_{c_1}+\cdots+\mathfrak{m}_{c_s}$.

An order $\mathfrak{m}=\{\Delta_1,\ldots,\Delta_k\}$
is called \emph{standard} if $\Delta_i\not\prec\Delta_j$ whenever  $i<j$.
Since $\Delta\prec\Delta'$ implies
\[
    \bb(\Delta)<\bb(\Delta')
    \qquad\text{and}\qquad
    \ee(\Delta)<\ee(\Delta'),
\]
the relation $\prec$ contains no directed cycles. Consequently,
standard orders always exist. We record the following elementary
constructions of standard orders, which will be used repeatedly.

\begin{lemma}\label{lem:orders}
Let $\mathfrak{m}\in\Oz$.
\begin{enumerate}[label=(\alph*)]
    \item Any order on $\mathfrak{m}$ in which the endpoints are
    non-increasing is standard. Likewise, any order in which the
    beginnings are non-increasing is standard.

    \item If $\Delta\in\mathfrak{m}$ has maximal endpoint, then there
    exists a standard order of $\mathfrak{m}$ whose first term is
    $\Delta$. If $\Delta\in\mathfrak{m}$ has minimal beginning, then
    there exists a standard order of $\mathfrak{m}$ whose last term
    is $\Delta$.

    \item Let $\mathfrak{n}\leq\mathfrak{m}$ and suppose that
    $\ee(\Gamma)>\ee(\Gamma')$
    for every $\Gamma\in\mathfrak{n}$ and every
    $\Gamma'\in\mathfrak{m}-\mathfrak{n}$. If $\sigma$ is a standard
    order of $\mathfrak{n}$ and $\rho$ is a standard order of
    $\mathfrak{m}-\mathfrak{n}$, then their concatenation
    $(\sigma,\rho)$ is a standard order of $\mathfrak{m}$.
\end{enumerate}
\end{lemma}

\begin{proof}
For part~\textup{(a)}, suppose that $\Delta$ occurs before $\Delta'$ in
an order whose endpoints are non-increasing. Then
\[
    \ee(\Delta)\geq\ee(\Delta').
\]
Since $\Delta\prec\Delta'$ would require
$\ee(\Delta)<\ee(\Delta')$, it follows that
$\Delta\not\prec\Delta'$. Hence, the order is standard. The proof for an
order with non-increasing beginnings is analogous, since
$\Delta\prec\Delta'$ requires
$\bb(\Delta)<\bb(\Delta')$.

For part~\textup{(b)}, suppose that $\Delta$ has maximal endpoint.
Then $\Delta\not\prec\Gamma$ for every
$\Gamma\in\mathfrak{m}$, because $\Delta\prec\Gamma$ would imply
\[
    \ee(\Delta)<\ee(\Gamma),
\]
contrary to the maximality of $\ee(\Delta)$. We may therefore place
$\Delta$ first and follow it by any standard order of
$\mathfrak{m}-\{\Delta\}$.

Dually, suppose that $\Delta$ has minimal beginning. Then
$\Gamma\not\prec\Delta$ for every $\Gamma\in\mathfrak{m}$, since
$\Gamma\prec\Delta$ would imply
\[
    \bb(\Gamma)<\bb(\Delta),
\]
contrary to the minimality of $\bb(\Delta)$. Thus, any standard order
of $\mathfrak{m}-\{\Delta\}$ may be followed by $\Delta$.

For part~\textup{(c)}, let $\Gamma\in\mathfrak{n}$ and
$\Gamma'\in\mathfrak{m}-\mathfrak{n}$. By assumption,
\[
    \ee(\Gamma)>\ee(\Gamma').
\]
Therefore, $\Gamma\not\prec\Gamma'$, because
$\Gamma\prec\Gamma'$ would require
$\ee(\Gamma)<\ee(\Gamma')$. Since $\sigma$ and $\rho$ are standard
within their respective sub-multisets, no pair of segments violates the
standardness condition in their concatenation. Hence, $(\sigma,\rho)$
is a standard order of $\mathfrak{m}$.
\end{proof}

\subsection{Relevant decompositions}

Let $\mathfrak{m}=\{\Delta_1,\ldots,\Delta_k\}$
be an order on a multiset of segments. A \emph{decomposition of the
ordered multiset} is a choice of a
decomposition $ (\Delta_{i,1},\ldots,\Delta_{i,k_i})$
of the segment $\Delta_i$ for each $i\in\{1,\ldots,k\}$. Let
$$I = \bigl\{(i,j):1\leq i\leq k,\ 1\leq j\leq k_i\bigr\}.$$
We equip $I$ with the lexicographic order, denoted by
$\preceq_{\mathrm{lex}}$, and with the partial order $\ll$ defined by
\[
    (i,j)\ll(i',j')
    \quad\Longleftrightarrow\quad
    i<i'.
\]
We refer to $i$ as the \emph{block} of the index $(i,j)$ and write
$\pr_1(i,j)=i$.
The decomposition is called \emph{trivial} if $k_i=1$ for every
$i\in\{1,\ldots,k\}$.
For $\imath=(i,j)\in I$, we use the abbreviated notation $\Delta_{\imath}=\Delta_{i,j}$.
\begin{definition}[{\cite[Definition~8.1]{MOS}}]\label{def:relevant}
A decomposition $ \{\Delta_{\imath}:\imath\in I\}$
is said to be \emph{relevant} to the order
$\{\Delta_1,\ldots,\Delta_k\}$ if there exists an involution
$\tau\colon I\to I$ satisfying the following conditions:
\begin{enumerate}[label=(R\arabic*)]
    \item $\tau(i,j+1)\ll\tau(i,j)$
    for every $1\leq i\leq k$ and for every $1\leq j\leq k_i-1$;

    \item $\tau(\imath)\neq\imath$ for every $\imath\in I$;

    \item $\Delta_{\imath}
        = \nu\Delta_{\tau(\imath)}$ whenever
    $\imath\prec_{\mathrm{lex}}\tau(\imath)$.
\end{enumerate}
\end{definition}

Given an involution $\tau$ as above, define $\varphi
    =
    \pr_1\circ\tau
    \colon
    I\longrightarrow\{1,\ldots,k\}$. 
Condition~\textup{(R1)} is equivalent to the requirement that
$\varphi$ be strictly decreasing along each block; explicitly,
\[
    \varphi(i,j+1)<\varphi(i,j)
    \qquad
    \text{for }1\leq j\leq k_i-1.
\]

\begin{definition}[{\cite[Definitions~8.3 and~8.4]{MOS}}]
\label{def:distspeh}
A multisegment $\mathfrak{m}\in\Oz$ is called \emph{distinguished} if
every standard order of $\mathfrak{m}$ admits a relevant decomposition.
It is said to be \emph{of Speh type} if there exists
$\mathfrak{n}\in\Oz$ such that
\[
    \mathfrak{m}
    =
    \mathfrak{n}+\nu\mathfrak{n}.
\]
\end{definition}

\begin{lemma}\label{lem:precnu}
For every $\Delta\in\Sgm$, one has $\Delta\prec\nu\Delta$. 
Consequently, in any standard order of $\mathfrak{m}$, every occurrence
of $\nu\Delta$ appears before every occurrence of $\Delta$.
\end{lemma}

\begin{proof}
Write $\Delta=[a,b]$. Then $\nu\Delta=[a+1,b+1]$. 
Since $a<a+1$, $b<b+1$, and $ b\geq a=(a+1)-1$, where  the inequality follows from $a\leq b$, we obtain $\Delta\prec\nu\Delta$. 
Now consider a standard order of $\mathfrak{m}$. If an occurrence of
$\Delta$ appeared before an occurrence of $\nu\Delta$, then the earlier
segment would precede the later segment, contradicting the definition
of a standard order. Hence, every occurrence of $\nu\Delta$ must appear
before every occurrence of $\Delta$.
\end{proof}

\begin{lemma}[{\cite[\S8.0.7]{MOS}}]\label{lem:trivialrel}
Let $\mathcal{O}=\{\Delta_1,\ldots,\Delta_k\}$
be an order on $\mathfrak{m}$.
\begin{enumerate}[label=(\alph*)]
    \item If the trivial decomposition is relevant to $\mathcal{O}$,
    then $\mathfrak{m}$ is of Speh type.

    \item If $\mathcal{O}$ is standard and $\mathfrak{m}$ is of Speh
    type, then the trivial decomposition is relevant to $\mathcal{O}$.
\end{enumerate}
Consequently, for a standard order $\mathcal{O}$, the trivial
decomposition is relevant to $\mathcal{O}$ if and only if
$\mathfrak{m}$ is of Speh type. In particular, every multisegment of
Speh type is distinguished.
\end{lemma}

\begin{proof}
For the trivial decomposition, one has $k_i=1$ for every $i$, and hence
\[
    I=\{(i,1):1\leq i\leq k\}.
\]
Identifying $(i,1)$ with $i$, we may regard $I$ as
$\{1,\ldots,k\}$. Condition~\textup{(R1)} is then vacuous, whereas
conditions~\textup{(R2)} and~\textup{(R3)} require a fixed-point-free
involution
\[
    \tau\colon\{1,\ldots,k\}\longrightarrow\{1,\ldots,k\}
\]
such that
\[
    \Delta_i=\nu\Delta_{\tau(i)}
    \qquad\text{whenever }i<\tau(i).
\]

Suppose first that the trivial decomposition is relevant to
$\mathcal{O}$. Since $\tau$ is a fixed-point-free involution, its
orbits are two-element sets. Define the multisegment
\[
    \mathfrak{n}
    =
    \bigl\{\Delta_{\tau(i)}:i<\tau(i)\bigr\}.
\]
For every pair $\{i,\tau(i)\}$ with $i<\tau(i)$, condition~\textup{(R3)}
gives
\[
    \Delta_i=\nu\Delta_{\tau(i)}.
\]
Thus, each orbit of $\tau$ contributes one segment to
$\mathfrak{n}$ and its $\nu$-translate to $\nu\mathfrak{n}$. Therefore,
\[
    \mathfrak{m}=\mathfrak{n}+\nu\mathfrak{n},
\]
so $\mathfrak{m}$ is of Speh type. Notice that this argument does not
require $\mathcal{O}$ to be standard.

Conversely, suppose that $\mathcal{O}$ is standard and that $\mathfrak{m}=\mathfrak{n}+\nu\mathfrak{n}$
for some $\mathfrak{n}\in\Oz$. We regard the occurrences in
$\mathfrak{m}$ as partitioned into two disjoint families: the
occurrences contributed by $\mathfrak{n}$ and those contributed by
$\nu\mathfrak{n}$. More precisely, for each segment $\Gamma$, choose
$\mathfrak{n}(\Gamma)$ occurrences of $\Gamma$ in $\mathfrak{m}$ and
label them as occurrences arising from $\mathfrak{n}$. Label the
remaining occurrences according to the summand $\nu\mathfrak{n}$.
This is possible because, for every segment $\Delta$,
\[
    \mathfrak{m}(\Delta)
    =
    \mathfrak{n}(\Delta)
    +
    \mathfrak{n}(\nu^{-1}\Delta).
\]
For every occurrence of a segment $\Gamma$ contributed by
$\mathfrak{n}$, pair it with the corresponding occurrence of
$\nu\Gamma$ contributed by $\nu\mathfrak{n}$. These pairs define a
fixed-point-free involution $\tau$ on $\{1,\ldots,k\}$.
By Lemma~\ref{lem:precnu}, one has $\Gamma\prec\nu\Gamma$. 
Since $\mathcal{O}$ is standard, every occurrence of $\nu\Gamma$
appears before every occurrence of $\Gamma$. Thus, if $i<\tau(i)$,
then $\Delta_i$ is the occurrence of $\nu\Gamma$ and
$\Delta_{\tau(i)}$ is the corresponding occurrence of $\Gamma$.
Consequently, $\Delta_i=\nu\Delta_{\tau(i)}$. 
Hence, condition~\textup{(R3)} holds. Conditions~\textup{(R1)} and
\textup{(R2)} are immediate, because the decomposition is trivial and
$\tau$ has no fixed points. Therefore, the trivial decomposition is
relevant to $\mathcal{O}$.

The equivalence now follows from parts~\textup{(a)} and~\textup{(b)}.
Finally, if $\mathfrak{m}$ is of Speh type, then the trivial
decomposition is relevant to every standard order of $\mathfrak{m}$.
Hence, $\mathfrak{m}$ is distinguished.
\end{proof}

\begin{remark}\label{rem:needstandard}
The standard assumption in Lemma~\ref{lem:trivialrel}\textup{(b)}
is essential. Correspondingly, the ``only if'' direction in
\cite[\S8.0.7]{MOS} is stated for standard orders.
Indeed, let $ \mathfrak{m}=\{\Delta,\nu\Delta\}$. 
Then $\mathfrak{m} = \{\Delta\}+\nu\{\Delta\}$, 
so $\mathfrak{m}$ is of Speh type. However, consider the order $\mathcal{O}=(\Delta,\nu\Delta)$. 
By Lemma~\ref{lem:precnu}, one has $\Delta\prec\nu\Delta$, and hence
$\mathcal{O}$ is not standard. The only fixed-point-free involution of
$\{1,2\}$ is the transposition $\tau=(1\ 2)$. Since $1<2$,
condition~\textup{(R3)} would require 
$\Delta = \nu(\nu\Delta) = \nu^2\Delta$, 
which is impossible for a finite segment. Therefore, the trivial
decomposition is not relevant to $\mathcal{O}$.
\end{remark}

\begin{remark}\label{rem:converse}
Lemma~\ref{lem:trivialrel} establishes the converse implication to
Hypothesis~\ref{hyp:85}: every multisegment of Speh type is
distinguished. Thus, Hypothesis~\ref{hyp:85} would provide the reverse
implication, and the two statements together would yield
$$\mathfrak{m}\text{ is distinguished}
    \quad\Longleftrightarrow\quad
    \mathfrak{m}\text{ is of Speh type}.$$

However, one should not conclude that the trivial decomposition is the
only relevant decomposition of every standard order of a multisegment
of Speh type. Indeed, consider
\[
    \mathfrak{m}
    =
    \{[0,0],[0,1],[1,1],[1,2]\}.
\]
This multisegment is of Speh type because 
$\mathfrak{m} = \mathfrak{n}+\nu\mathfrak{n}$ where $\mathfrak{n}=\{[0,0],[0,1]\}$. 
The standard order
\[
    \mathcal{O}
    =
    ([1,1],[1,2],[0,0],[0,1])
\]
admits, in addition to the trivial decomposition, the non-trivial
decomposition
$$\Delta_{1}=[1,1],\quad \Delta_{2}=[1,2]=([2,2],[1,1]),\quad \Delta_{3}=[0,0],\quad
\Delta_{4}=[0,1]=([1,1],[0,0]).$$
In terms of the corresponding index set 
$I = \{(1,1),(2,1),(2,2),(3,1),(4,1),(4,2)\}$, 
define $\tau$ by exchanging the elements in each of the pairs
\[
    \{(1,1),(4,2)\},\qquad
    \{(2,1),(4,1)\},\qquad
    \{(2,2),(3,1)\}.
\]
Condition~\textup{(R1)} holds because $\varphi(2,1)=4>3=\varphi(2,2)$ and $ \varphi(4,1)=2>1=\varphi(4,2)$. 
Moreover, for the lexicographically smaller index in each pair, one has
\[
\begin{aligned}
    \Delta_{1,1}
        &=[1,1]=\nu[0,0]=\nu\Delta_{4,2},\\
    \Delta_{2,1}
        &=[2,2]=\nu[1,1]=\nu\Delta_{4,1},\\
    \Delta_{2,2}
        &=[1,1]=\nu[0,0]=\nu\Delta_{3,1}.
\end{aligned}
\]
Thus, Condition~\textup{(R3)} also holds, and the decomposition is
relevant.

Consequently, the property that the trivial decomposition is the only
relevant decomposition is a property of a particular order, rather
than of the multisegment itself. For the multisegment above, the canonical order $([1,2],[1,1],[0,1],[0,0])$
of \cite[\S8.0.12]{MOS} admits only the trivial relevant decomposition.
\end{remark}

By Lemma~\ref{lem:trivialrel}, to prove Hypothesis~\ref{hyp:85} for a
given multisegment $\mathfrak{m}$, it is enough to exhibit a standard
order of $\mathfrak{m}$ that admits no non-trivial relevant
decomposition. Indeed, if $\mathfrak{m}$ is distinguished, then this
order must admit a relevant decomposition; it must therefore admit the
trivial decomposition, and Lemma~\ref{lem:trivialrel}\textup{(a)}
implies that $\mathfrak{m}$ is of Speh type. This is the strategy
employed in \cite[\S8.0.10]{MOS}.

We call such an order a \emph{witness} for $\mathfrak{m}$. More
generally, a standard order that admits no relevant decomposition at
all is called a \emph{strong witness}. It follows directly from the
definition that a multisegment is non-distinguished if and only if it
admits a strong witness.

\section{A numerical criterion for multisegments of Speh type}\label{sec:speh}

For each fixed cardinality, the set of segments forms a single
$\nu$-orbit. Consequently, the property of being of Speh type may be
examined separately on each such orbit. The following numerical
criterion makes this observation precise and will be used repeatedly.

\begin{definition}\label{def:S}
Let $\mathfrak{m}\in\Oz$ and $\Delta\in\Sgm$. Define
\[
    S_{\mathfrak{m}}(\Delta)
    :=
    \sum_{t\geq 0}(-1)^t
    \mathfrak{m}(\nu^t\Delta) =
    \mathfrak{m}(\Delta)
    -
    \mathfrak{m}(\nu\Delta)
    +
    \mathfrak{m}(\nu^2\Delta)
    -\cdots.
\]
This sum is finite because $\mathfrak{m}$ has finite support.
\end{definition}

It follows immediately from the definition that for $\Delta\in\Sgm$ we have 
\begin{equation}\label{eq:Srec}
    S_{\mathfrak{m}}(\Delta)
    +
    S_{\mathfrak{m}}(\nu\Delta)
    =
    \mathfrak{m}(\Delta).
\end{equation}

\begin{theorem}\label{thm:speh}
A multisegment $\mathfrak{m}\in\Oz$ is of Speh type if and only if $S_{\mathfrak{m}}(\Delta)\geq 0$
for every $\Delta\in\Sgm$.
\end{theorem}

\begin{proof}
Suppose first that $\mathfrak{m}$ is of Speh type. Then there exists
$\mathfrak{n}\in\Oz$ such that $\mathfrak{m}=\mathfrak{n}+\nu\mathfrak{n}$. 
By the definition of the action of $\nu$ on multisegments, we have 
$(\nu\mathfrak{n})(\nu^t\Delta) =
    \mathfrak{n}(\nu^{t-1}\Delta)$. 
Therefore,
\begin{align*}
    S_{\mathfrak{m}}(\Delta)
    &=
    \sum_{t\geq 0}(-1)^t
    \mathfrak{n}(\nu^t\Delta)
    +
    \sum_{t\geq 0}(-1)^t
    \mathfrak{n}(\nu^{t-1}\Delta)\\
    &=
    \sum_{t\geq 0}(-1)^t
    \mathfrak{n}(\nu^t\Delta)
    +
    \mathfrak{n}(\nu^{-1}\Delta)
    +
    \sum_{t\geq 1}(-1)^t
    \mathfrak{n}(\nu^{t-1}\Delta).
\end{align*}
Reindexing the last sum by $s=t-1$ gives
\[
    \sum_{t\geq 1}(-1)^t
    \mathfrak{n}(\nu^{t-1}\Delta)
    =
    -
    \sum_{s\geq 0}(-1)^s
    \mathfrak{n}(\nu^s\Delta).
\]
Hence, the two alternating sums cancel and we obtain $S_{\mathfrak{m}}(\Delta) =
    \mathfrak{n}(\nu^{-1}\Delta) \geq 0$. 

Conversely, suppose that $S_{\mathfrak{m}}(\Delta)\geq 0$ for every $\Delta\in\Sgm$. 
Define a function $\mathfrak{n}\colon\Sgm\to\Zpos$ by
\[
    \mathfrak{n}(\Delta)
    :=
    S_{\mathfrak{m}}(\nu\Delta).
\]
Using \eqref{eq:Srec}, we obtain
\begin{align*}
    (\mathfrak{n}+\nu\mathfrak{n})(\Delta)
    &=
    \mathfrak{n}(\Delta)
    +
    \mathfrak{n}(\nu^{-1}\Delta)\\
    &=
    S_{\mathfrak{m}}(\nu\Delta)
    +
    S_{\mathfrak{m}}(\Delta)\\
    &=
    \mathfrak{m}(\Delta).
\end{align*}
It remains only to verify that $\mathfrak{n}$ has finite support.

Only finitely many segment cardinalities occur in support of $\mathfrak{m}$. For every other cardinality, all multiplicities of
$\mathfrak{m}$, and hence all corresponding values of
$S_{\mathfrak{m}}$, vanish. We may therefore fix a cardinality
$\ell$ occurring in $\mathfrak{m}$ and set
$$c_a :=
    \mathfrak{m}([a,a+\ell-1]),
    \qquad
    S_a :=
    S_{\mathfrak{m}}([a,a+\ell-1]).$$
Equation~\eqref{eq:Srec} gives $S_a+S_{a+1}=c_a$, 
or equivalently, $S_a=c_a-S_{a+1}$. 
For all sufficiently large $a$, one has $c_{a'}=0$  for every $a'\geq a$. 
It follows directly from the definition of $S_{\mathfrak{m}}$ that
$S_a=0$ for all such $a$.
Similarly, for all sufficiently small $a$, one has $c_a=0$. Hence, $S_a+S_{a+1}=0$. 
Since both $S_a$ and $S_{a+1}$ are nonnegative, this implies $S_a=S_{a+1}=0$. 
Thus, for each relevant cardinality $\ell$, only finitely many of the
values $S_a$ are nonzero. Since only finitely many cardinalities occur
in $\mathfrak{m}$, the function $\mathfrak{n}$ has finite support and
therefore belongs to $\Oz$. Consequently, $\mathfrak{m} = \mathfrak{n}+\nu\mathfrak{n}$, 
so $\mathfrak{m}$ is of Speh type.
\end{proof}

\begin{remark}\label{rem:peelingnumbers}
Let $c$ be the largest endpoint occurring among the segments of
$\mathfrak{m}$. For each $x\leq c$, define a finitely supported
integer-valued function $R_x$ on $\Sgm$, supported on the segments
having endpoint $x$, by the recursion
$$R_c=\mathfrak{m}_c,
    \qquad
    R_{x-1} =
    \mathfrak{m}_{x-1}-\nu^{-1}R_x.$$
Unwinding this recursion yields 
$R_x(\Delta)  = S_{\mathfrak{m}}(\Delta)$ whenever $\ee(\Delta)=x$. 
Thus, $R_x$ records the signed multiplicities remaining after
successively peeling $\mathfrak{m}$ from the largest endpoint
downwards. In particular, $R_x$ is a genuine multiset precisely when
all its coefficients are nonnegative. Therefore, 
Theorem~\ref{thm:speh} states that $\mathfrak{m}$ is of Speh
type if and only if $R_x\geq 0$ for every $x\leq c$. 
The first nontrivial condition in this procedure is $R_{c-1}  = \mathfrak{m}_{c-1} - \nu^{-1}\mathfrak{m}_c
    \geq 0$,
which is equivalent to 
$$\nu^{-1}\mathfrak{m}_c \leq \mathfrak{m}_{c-1}.$$
\end{remark}

\section{Rigidity properties of relevant decompositions}\label{sec:rigid}

Throughout this section, we fix a standard order $\mathcal{O}=\{\Delta_1,\ldots,\Delta_k\}$
of a multisegment $\mathfrak{m}$, a relevant decomposition $ \{\Delta_{\imath}\}_{\imath\in I}$, 
and an involution $\tau$ satisfying conditions
\textup{(R1)}--\textup{(R3)} of Definition~\ref{def:relevant}. We also
write $\varphi=\pr_1\circ\tau$. 

\begin{lemma}\label{lem:blocks}
For every $(i,j)\in I$, one has $\varphi(i,j)\neq i$. 
Equivalently, $\tau$ never maps an index to another index belonging to
the same block.
\end{lemma}

\begin{proof}
Suppose, to the contrary, that $\tau(i,j)=(i,j')$
for some $j'$. Since $\tau$ is an involution, it follows that $\tau(i,j')=(i,j)$. Consequently, we have $ \varphi(i,j)=\varphi(i,j')=i$. 
On the other hand, condition~\textup{(R1)} implies that $\varphi$ is
strictly decreasing, and hence injective, along the $i$-th block.
Therefore, $j=j'$. It follows that $\tau(i,j)=(i,j)$, contradicting
condition~\textup{(R2)}.
\end{proof}

The following lemma combines \cite[Lemma~8.2]{MOS} with its dual. We
include the short proofs for completeness.

\begin{lemma}\label{lem:firstlast}
\noindent
\begin{enumerate}[label=(\alph*)]
 \item There exist indices $k\geq i_1>i_2>\cdots>i_{k_1}>1$
    such that $\tau(1,j)=(i_j,k_{i_j})$, where $1\leq j\leq k_1$. 
    Thus, the pieces of the first block are paired with the bottom
    pieces of later blocks.

    \item There exist indices $1\leq i'_{k_k}<\cdots<i'_2<i'_1<k$
    such that $\tau(k,j)=(i'_j,1)$, where $1\leq j\leq k_k$. Thus, the pieces of the last block are paired with the top pieces
    of earlier blocks.
\end{enumerate}
\end{lemma}

\begin{proof}
For part~\textup{(a)}, write $\tau(1,j)=(i_j,r_j)$. 
By condition~\textup{(R1)}, the map $\varphi$ is strictly decreasing
along the first block. Hence, we have $i_1>i_2>\cdots>i_{k_1}$. 
Moreover, Lemma~\ref{lem:blocks} implies that $i_j\neq 1$ for every
$j$, and therefore $i_{k_1}>1$. 
We claim that $r_j=k_{i_j}$. Suppose instead that
$r_j<k_{i_j}$. Applying condition~\textup{(R1)} to the $i_j$-th block
at the index $r_j$, we obtain
\[
    \tau(i_j,r_j+1)\ll\tau(i_j,r_j).
\]
Since $\tau$ is an involution and
$\tau(1,j)=(i_j,r_j)$, we have $ \tau(i_j,r_j)=(1,j)$. 
It follows that
\[
    \varphi(i_j,r_j+1)
    <
    \varphi(i_j,r_j)
    =
    1,
\]
which is impossible. Therefore, we get $r_j=k_{i_j}$, as required.

For part~\textup{(b)}, write $ \tau(k,j)=(i'_j,r'_j)$. 
Condition~\textup{(R1)} gives $ i'_1>i'_2>\cdots>i'_{k_k}$. By Lemma~\ref{lem:blocks}, one has $i'_j\neq k$. Since
$i'_j\leq k$, it follows that $i'_j<k$. 
We claim that $r'_j=1$. Suppose instead that $r'_j>1$. Applying
condition~\textup{(R1)} to the $i'_j$-th block at the index
$r'_j-1$, we obtain
\[
    \varphi(i'_j,r'_j)
    <
    \varphi(i'_j,r'_j-1).
\]
Since $\tau$ is an involution and
$\tau(k,j)=(i'_j,r'_j)$, we have $ \tau(i'_j,r'_j)=(k,j)$, 
and hence $\varphi(i'_j,r'_j)=k$. 
Therefore, $k < \varphi(i'_j,r'_j-1) \leq k$, 
a contradiction. Thus, we have $r'_j=1$.
\end{proof}

We next extract the numerical consequences of
Lemma~\ref{lem:firstlast}. Recall that $\Delta_{i,1}$ is the top piece
and $\Delta_{i,k_i}$ is the bottom piece of $\Delta_i$. Therefore, $\ee(\Delta_{i,1})=\ee(\Delta_i)$,  $\bb(\Delta_{i,k_i})=\bb(\Delta_i)$, and 
$$\bb(\Delta_{i,1})
    >
    \bb(\Delta_{i,2})
    >
    \cdots
    >
    \bb(\Delta_{i,k_i}).$$

\begin{lemma}\label{lem:first}
With the notation of Lemma~\ref{lem:firstlast}\textup{(a)}, for every
$j\in\{1,\ldots,k_1\}$, one has $\Delta_{i_j,k_{i_j}} =
    \nu^{-1}\Delta_{1,j}$, and consequently $\bb(\Delta_{i_j})
    =
    \bb(\Delta_{1,j})-1$ and $\ee(\Delta_{i_j})
    \geq
    \ee(\Delta_{1,j})-1$.
In particular, $\bb(\Delta_{i_1})
    >
    \bb(\Delta_{i_2})
    >
    \cdots
    >
    \bb(\Delta_{i_{k_1}})
    =
    \bb(\Delta_1)-1$, 
and $\ee(\Delta_{i_1})
    \geq
    \ee(\Delta_1)-1$. 
In particular, $\mathfrak{m}$ contains a segment whose beginning is
$\bb(\Delta_1)-1$.
\end{lemma}

\begin{proof}
Since $i_j>1$, we have $(1,j)\prec_{\mathrm{lex}}(i_j,k_{i_j}) =
    \tau(1,j)$. Therefore, Condition~\textup{(R3)} gives $ \Delta_{1,j} = \nu\Delta_{i_j,k_{i_j}}$, 
or equivalently, $ \Delta_{i_j,k_{i_j}} = \nu^{-1}\Delta_{1,j}$. Since $\Delta_{i_j,k_{i_j}}$ is the bottom piece of $\Delta_{i_j}$,
we obtain
\begin{align*}
    \bb(\Delta_{i_j})
    &=
    \bb(\Delta_{i_j,k_{i_j}})
    =
    \bb(\Delta_{1,j})-1,\\
    \ee(\Delta_{i_j})
    &\geq
    \ee(\Delta_{i_j,k_{i_j}})
    =
    \ee(\Delta_{1,j})-1.
\end{align*}
Because the beginnings of the pieces of $\Delta_1$ strictly decrease
with $j$, it follows that
\[
    \bb(\Delta_{i_1})
    >
    \cdots
    >
    \bb(\Delta_{i_{k_1}}).
\]
Finally, $\bb(\Delta_{1,k_1})=\bb(\Delta_1)$ and $\ee(\Delta_{1,1})=\ee(\Delta_1)$, 
which yield the remaining assertions.
\end{proof}

\begin{lemma}\label{lem:last}
With the notation of Lemma~\ref{lem:firstlast}\textup{(b)}, for every
$j\in\{1,\ldots,k_k\}$, one has $ \Delta_{i'_j,1} = \nu\Delta_{k,j}$, and consequently 
$\ee(\Delta_{i'_j}) = \ee(\Delta_{k,j})+1$ and $\bb(\Delta_{i'_j})
    \leq
    \bb(\Delta_{k,j})+1$. 
In particular, $\ee(\Delta_{i'_1}) = \ee(\Delta_k)+1$.
Thus, $\mathfrak{m}$ contains a segment whose endpoint is
$\ee(\Delta_k)+1$.
\end{lemma}

\begin{proof}
Since $i'_j<k$, we have $ (i'_j,1) \prec_{\mathrm{lex}} (k,j)$. 
Moreover, $\tau(i'_j,1)=(k,j)$,
because $\tau$ is an involution. Therefore, Condition~\textup{(R3)} gives $\Delta_{i'_j,1} =
    \nu\Delta_{k,j}$. Since $\Delta_{i'_j,1}$ is the top piece of $\Delta_{i'_j}$, we obtain
\begin{align*}
    \ee(\Delta_{i'_j})
    &=
    \ee(\Delta_{i'_j,1})
    =
    \ee(\Delta_{k,j})+1,\\
    \bb(\Delta_{i'_j})
    &\leq
    \bb(\Delta_{i'_j,1})
    =
    \bb(\Delta_{k,j})+1.
\end{align*}
For $j=1$, the segment $\Delta_{k,1}$ is the top piece of $\Delta_k$, so $\ee(\Delta_{k,1})=\ee(\Delta_k)$. 
It follows that $\ee(\Delta_{i'_1}) = \ee(\Delta_k)+1$.
\end{proof}

Lemmas~\ref{lem:first} and~\ref{lem:last} have the following global
consequences.

\begin{corollary}\label{cor:sameend}
Let $\mathfrak{m}\neq 0$.
\begin{enumerate}[label=(\alph*)]
    \item If all segments of $\mathfrak{m}$ have the same endpoint,
    then no standard order of $\mathfrak{m}$ admits a relevant
    decomposition. The same conclusion holds if all segments of
    $\mathfrak{m}$ have the same beginning.

    \item If $\mathfrak{m}$ is distinguished, then the following
    statements hold:
    \begin{enumerate}[label=(\roman*)]
        \item for every segment $\Gamma\in\mathfrak{m}$ having maximal
        endpoint, $\mathfrak{m}$ contains a segment whose beginning is
        $\bb(\Gamma)-1$;

        \item for every segment $\Gamma\in\mathfrak{m}$ having minimal
        beginning, $\mathfrak{m}$ contains a segment whose endpoint is
        $\ee(\Gamma)+1$.
    \end{enumerate}
\end{enumerate}
\end{corollary}

\begin{proof}
Suppose first that all segments of $\mathfrak{m}$ have the same
endpoint $c$. If some standard order of $\mathfrak{m}$ admitted a
relevant decomposition, then Lemma~\ref{lem:last}, applied to its last
block, would produce a segment of $\mathfrak{m}$ with endpoint $c+1$.
This contradicts the assumption that all endpoints are equal to $c$.
Similarly, suppose that all segments of $\mathfrak{m}$ have the same
beginning $a$. If some standard order admitted a relevant
decomposition, then Lemma~\ref{lem:first}, applied to its first block,
would produce a segment of $\mathfrak{m}$ with beginning $ \bb(\Delta_1)-1=a-1$, 
again a contradiction. This proves part~\textup{(a)}.

For part~\textup{(b)}, let $\Gamma\in\mathfrak{m}$ have a maximal endpoint. By Lemma~\ref{lem:orders}\textup{(b)}, there exists a standard
order of $\mathfrak{m}$ beginning with $\Gamma$. Since
$\mathfrak{m}$ is distinguished, this order admits a relevant
decomposition. Lemma~\ref{lem:first} then produces a segment of
$\mathfrak{m}$ whose beginning is $\bb(\Gamma)-1$.

The second assertion is dual. If $\Gamma$ has minimal beginning, then
Lemma~\ref{lem:orders}\textup{(b)} provides a standard order ending
with $\Gamma$. Since this order admits a relevant decomposition,
Lemma~\ref{lem:last} produces a segment of $\mathfrak{m}$ whose endpoint
is $\ee(\Gamma)+1$.
\end{proof}

\section{Duality and invariance of relevance}\label{sec:duality}

Recall that $\Delta^\vee = [-\ee(\Delta),-\bb(\Delta)]$
and that the dual multisegment $\mathfrak{m}^\vee$ is defined by  $$\mathfrak{m}^\vee(\Delta) =
    \mathfrak{m}(\Delta^\vee).$$
\begin{lemma}\label{lem:dualprec}
Let $\Delta,\Delta'\in\Sgm$. Then $\Delta\prec\Delta'$ if and only if $(\Delta')^\vee\prec\Delta^\vee$. 
\end{lemma}

\begin{proof}
Write $\Delta=[a,b]$ and $\Delta'=[a',b']$. 
Then we have $ (\Delta')^\vee=[-b',-a']$ and $\Delta^\vee=[-b,-a]$. The relation $(\Delta')^\vee\prec\Delta^\vee$
is equivalent to $-b'<-b$, $-a'<-a$, and $-a'\geq -b-1$. These inequalities are equivalent, respectively, to
$b<b'$, $a<a'$, and $b\geq a'-1$, 
which are precisely the conditions defining
$\Delta\prec\Delta'$.
\end{proof}

\begin{lemma}\label{lem:dualdec}
A tuple $(\Delta_1,\ldots,\Delta_k)$ is a decomposition of $\Delta$ if
and only if $ (\Delta_k^\vee,\ldots,\Delta_1^\vee)$
is a decomposition of $\Delta^\vee$.
\end{lemma}

\begin{proof}
Write $\Delta_i=[a_i,b_i]$. 
Since $(\Delta_1,\ldots,\Delta_k)$ is a decomposition of $\Delta$, one has $b_1=\ee(\Delta$, $ a_k=\bb(\Delta)$, and $b_{i+1}=a_i-1$, where $1\leq i\leq k-1$. The reversed dual tuple is
\[
    ([-b_k,-a_k],\ldots,[-b_1,-a_1]).
\]
Its first term has an endpoint $-a_k=-\bb(\Delta)=\ee(\Delta^\vee)$, 
and its last term has a beginning $ -b_1=-\ee(\Delta)=\bb(\Delta^\vee)$. 
For $1\leq r\leq k-1$, the endpoint of the $(r+1)$-st term is 
$$-a_{k-r} = -(b_{k-r+1}+1) = -b_{k-r+1}-1,$$
which is one less than the beginning of the $r$-th term. Hence, the
reversed dual tuple is a decomposition of $\Delta^\vee$.
The converse follows by applying the same argument to $\Delta^\vee$ and using $ \Delta^{\vee\vee}=\Delta$.
\end{proof}

\begin{theorem}\label{thm:duality}
Let $\mathfrak{m}\in\Oz$, and let $ \mathcal{O}=(\Delta_1,\ldots,\Delta_k)$
be an order on $\mathfrak{m}$. Define $ \mathcal{O}^\vee := (\Delta_k^\vee,\ldots,\Delta_1^\vee)$, 
which is an order on $\mathfrak{m}^\vee$. Then the following statements
hold:
\begin{enumerate}[label=(\alph*)]
    \item The order $\mathcal{O}$ is standard if and only if
    $\mathcal{O}^\vee$ is standard. Moreover, the assignment $\mathcal{O}\longmapsto\mathcal{O}^\vee$
    defines a bijection between the standard orders of $\mathfrak{m}$
    and those of $\mathfrak{m}^\vee$.

    \item The order $\mathcal{O}$ admits a relevant decomposition if
    and only if $\mathcal{O}^\vee$ admits a relevant decomposition.
    This correspondence preserves triviality.
\end{enumerate}
Consequently, $\mathfrak{m}$ is distinguished if and only if $\mathfrak{m}^\vee$ is distinguished. 
\end{theorem}

\begin{proof}
For part~\textup{(a)}, write $\Gamma_p=\Delta_{k+1-p}^\vee$ for $1\leq p\leq k$, so that $\mathcal{O}^\vee=(\Gamma_1,\ldots,\Gamma_k)$. 
If $p<q$, then Lemma~\ref{lem:dualprec} gives $\Gamma_p\prec\Gamma_q$ if and only if $ \Delta_{k+1-q}\prec\Delta_{k+1-p}$. 
Since $k+1-q<k+1-p$, 
the standard condition for $\mathcal{O}^\vee$ is equivalent to
that for $\mathcal{O}$. Moreover, the assignment
$\mathcal{O}\mapsto\mathcal{O}^\vee$ is involutive and is therefore a
bijection.

For part~\textup{(b)}, suppose that $\mathcal{O}$ admits a relevant decomposition $\{\Delta_{i,j}:(i,j)\in I\}$, 
and let $\tau$ be an involution satisfying conditions
\textup{(R1)}--\textup{(R3)}.
By Lemma~\ref{lem:dualdec}, the reversed dual tuples $(\Delta_{i,k_i}^\vee,\ldots,\Delta_{i,1}^\vee)$
define a decomposition of $\mathcal{O}^\vee$. In this decomposition,
the block corresponding to $\Delta_i^\vee$ has index $k+1-i$ and contains $k_i$ pieces.
Let $I^\vee$ denote the corresponding index set, and define
\[
    \psi\colon I\longrightarrow I^\vee,
    \qquad
    \psi(i,j)
    =
    (k+1-i,k_i+1-j).
\]
If $\widetilde{\Delta}_{\psi(i,j)}$ denotes the piece of the dual
decomposition indexed by $\psi(i,j)$, then $\widetilde{\Delta}_{\psi(i,j)} = \Delta_{i,j}^\vee$. 
Furthermore, $\pr_1\circ\psi = k+1-\pr_1$, 
and $\psi$ reverses the lexicographic order.
Define
$$\tau^\vee := \psi\circ\tau\circ\psi^{-1}.$$
Then $\tau^\vee$ is an involution of $I^\vee$. Since $\tau$ has no
fixed points and $\psi$ is bijective, $\tau^\vee$ has no fixed points.
Thus, Condition~\textup{(R2)} holds.

\smallskip
\noindent\emph{Verification of condition~\textup{(R1)}.}
Fix a block $i'=k+1-i$
of the dual decomposition and let $1\leq j'\leq k_i-1$. Put $u=k_i+1-j'$. Then
\[
    \psi^{-1}(i',j')=(i,u),
    \qquad
    \psi^{-1}(i',j'+1)=(i,u-1).
\]
Writing $\varphi=\pr_1\circ\tau$, we obtain
\begin{align*}
    &\pr_1\bigl(\tau^\vee(i',j'+1)\bigr)
    <
    \pr_1\bigl(\tau^\vee(i',j')\bigr)\\
    &\quad\Longleftrightarrow\quad
    k+1-\varphi(i,u-1)
    <
    k+1-\varphi(i,u)\\
    &\quad\Longleftrightarrow\quad
    \varphi(i,u)<\varphi(i,u-1).
\end{align*}
The last inequality is precisely condition~\textup{(R1)} for $\tau$
at the index $(i,u-1)$.

\smallskip
\noindent\emph{Verification of condition~\textup{(R3)}.}
Let $\imath\in I^\vee$ satisfy $\imath
    \prec_{\mathrm{lex}}
    \tau^\vee(\imath)$, 
and write $\imath=\psi(x)$ for some $x\in I$. Since $\tau^\vee(\psi(x))=\psi(\tau(x))$
and $\psi$ reverses the lexicographic order, the preceding inequality
is equivalent to $\tau(x)\prec_{\mathrm{lex}}x$. 
Applying condition~\textup{(R3)} for $\tau$ to the index $\tau(x)$ gives $\Delta_{\tau(x)} = \nu\Delta_x$. Using $(\nu\Delta_x)^\vee =
    \nu^{-1}\Delta_x^\vee$, we obtain
$$\nu\widetilde{\Delta}_{\tau^\vee(\imath)} = \nu\widetilde{\Delta}_{\psi(\tau(x))} = \nu\Delta_{\tau(x)}^\vee = \nu(\nu\Delta_x)^\vee = \nu\nu^{-1}\Delta_x^\vee = \Delta_x^\vee = \widetilde{\Delta}_{\imath}.$$
Thus, Condition~\textup{(R3)} holds for $\tau^\vee$. Therefore, the
dual decomposition is relevant to $\mathcal{O}^\vee$.

The construction is involutive, so the converse follows by applying
the same argument to $\mathcal{O}^\vee$. Moreover, each block and its
corresponding dual block have the same number of pieces. Hence, the
correspondence preserves triviality.

Finally, part~\textup{(a)} gives a bijection between the standard orders
of $\mathfrak{m}$ and those of $\mathfrak{m}^\vee$, while
part~\textup{(b)} preserves the existence of relevant decompositions.
The final assertion therefore follows from the definition of
distinguishedness.
\end{proof}

\begin{lemma}\label{lem:spehdual}
A multisegment $\mathfrak{m}$ is of Speh type if and only if
$\mathfrak{m}^\vee$ is of Speh type.
\end{lemma}
\begin{proof}
Suppose that $\mathfrak{m} = \mathfrak{n}+\nu\mathfrak{n}$. Since $(\nu\mathfrak{n})^\vee = \nu^{-1}\mathfrak{n}^\vee$, 
we have $\mathfrak{m}^\vee =
\mathfrak{n}^\vee+\nu^{-1}\mathfrak{n}^\vee$. 
Set $\mathfrak{p} =
    \nu^{-1}\mathfrak{n}^\vee$. 
Then $\nu\mathfrak{p} =
    \mathfrak{n}^\vee$, and hence $ \mathfrak{m}^\vee =
    \mathfrak{p}+\nu\mathfrak{p}$. 
Thus, $\mathfrak{m}^\vee$ is of Speh type. The converse follows by
applying the same argument to $\mathfrak{m}^\vee$ and using
$\mathfrak{m}^{\vee\vee}=\mathfrak{m}$.
\end{proof}

\begin{corollary}\label{cor:85iff86}
Hypotheses~\ref{hyp:85} and~\ref{hyp:86} are equivalent. More
precisely, for each fixed multisegment $\mathfrak{m}$,
Hypothesis~\ref{hyp:85} holds for $\mathfrak{m}$ if and only if
Hypothesis~\ref{hyp:86} holds for $\mathfrak{m}$.
\end{corollary}

\begin{proof}
Hypothesis~\ref{hyp:85} asserts that if $\mathfrak{m}$ is distinguished, then $\mathfrak{m}$ is of Speh type, 
whereas Hypothesis~\ref{hyp:86} asserts that if $\mathfrak{m}$ and $\mathfrak{m}^\vee$ are both distinguished, then $\mathfrak{m}$ is of Speh type. 
By Theorem~\ref{thm:duality}, $\mathfrak{m}$ is distinguished if and only if $\mathfrak{m}^\vee$ is distinguished. Consequently, $\mathfrak{m}$ is distinguished precisely when both $\mathfrak{m}$ and $\mathfrak{m}^\vee$ are distinguished.
Thus, the antecedents of Hypotheses~\ref{hyp:85} and~\ref{hyp:86}
are equivalent, while their conclusions are identical.
\end{proof}

\begin{corollary}\label{cor:dual88}
Hypotheses~\ref{hyp:85} and~\ref{hyp:86} hold for every
$\mathfrak{m}\in\Oz$ satisfying $\lvert\mathfrak{m}^x\rvert\leq 2$ for every $x\in\Z$. 
Equivalently, at most two segments of $\mathfrak{m}$, counted with
multiplicity, have any prescribed beginning.
\end{corollary}

\begin{proof}
For every $x\in\Z$, one has $\lvert\mathfrak{m}^x\rvert =
    \lvert(\mathfrak{m}^\vee)_{-x}\rvert$. 
Indeed, duality sends a segment with beginning $x$ to a segment with
endpoint $-x$. Therefore, $\mathfrak{m}^\vee$ satisfies the assumption
of \cite[Proposition~8.8]{MOS}. Suppose that $\mathfrak{m}$ is distinguished. By
Theorem~\ref{thm:duality}, $\mathfrak{m}^\vee$ is distinguished.
It follows from \cite[Proposition~8.8]{MOS} that
$\mathfrak{m}^\vee$ is of Speh type. Lemma~\ref{lem:spehdual} then
implies that $\mathfrak{m}$ is of Speh type. Thus,
Hypothesis~\ref{hyp:85} holds for $\mathfrak{m}$, and
Corollary~\ref{cor:85iff86} shows that Hypothesis~\ref{hyp:86} holds
as well.
\end{proof}

\begin{remark}
More generally, let $\mathcal{C}$ be any class of multisegments for which Hypothesis~\ref{hyp:85} is known, and define $ \mathcal{C}^\vee := \{\mathfrak{m}^\vee:\mathfrak{m}\in\mathcal{C}\}$.
Theorem~\ref{thm:duality}, together with
Lemma~\ref{lem:spehdual}, implies that
Hypothesis~\ref{hyp:85} also holds for every multisegment in $\mathcal{C}^\vee$. Since $\mathfrak{m}$ is a set if and only if $\mathfrak{m}^\vee$ is a
set, the class considered in \cite[Proposition~8.7]{MOS} is self-dual.
By contrast, the condition in \cite[Proposition~8.8]{MOS}, which is
formulated in terms of common endpoints, is transformed by duality
into the condition involving common beginnings in
Corollary~\ref{cor:dual88}. Thus, Corollary~\ref{cor:dual88} provides
the corresponding dual class.
\end{remark}

\section{Peeling and extremal layers of multisegments}\label{sec:peel}
Throughout this section, we assume that $\mathfrak{m}\ne 0$ and write $c$ for the largest endpoint occurring in $\mathfrak{m}$. We use the notation $\mathfrak{m}_{c}$ and $\mathfrak{m}_{c-1}$ from Section~\ref{sec:setting}; the latter may be zero.

\begin{definition}\label{def:admissible}
A listing $\sigma$ of $\mathfrak{m}_{c-1}$ is \emph{admissible for $\mathfrak{m}$} if it satisfies the following conditions:
\begin{enumerate}[label=(A\arabic*)]
\item If $B\in\supp(\nu^{-1}\mathfrak{m}_{c})$, every occurrence of $B$ in $\sigma$ follows every occurrence of a segment $\Gamma\in\mathfrak{m}_{c-1}$ with $\bb(\Gamma)>\bb(B)$.
\item All occurrences of any given segment occupy consecutive positions in $\sigma$.
\end{enumerate}
\end{definition}
For example, any listing of $\mathfrak{m}_{c-1}$ by non-increasing beginnings is admissible.
List $\mathfrak{m}_{c}$ as $A_{1},\ldots,A_{r}$ in non-decreasing order of their beginnings. Let $\sigma$ be an admissible listing of $\mathfrak{m}_{c-1}$, and let $\rho$ be any standard order of
$\mathfrak{m}-\mathfrak{m}_{c}-\mathfrak{m}_{c-1}$. Then by Lemma~\ref{lem:orders}(a),(c) the concatenation
\begin{equation}\label{eq:peelorder}
\mathcal{O}=(A_{1},\ldots,A_{r},\ \sigma,\ \rho)
\end{equation}
is a standard order of $\mathfrak{m}$. Indeed, segments with the same endpoint cannot precede one another, while a segment in an earlier endpoint layer cannot precede one in a later layer; the remaining entries are standard by choice of $\rho$. 
\begin{theorem}\label{thm:peel}
Suppose that the standard order $\mathcal{O}$ in \eqref{eq:peelorder}
admits a relevant decomposition. Then $\nu^{-1}\mathfrak{m}_{c}\le \mathfrak{m}_{c-1}$. 
Moreover, the standard order $(\sigma',\rho)$ of
$\mathfrak{m}-\mathfrak{m}_{c}-\nu^{-1}\mathfrak{m}_{c}$
admits a relevant decomposition, where $\sigma'$ is obtained from
$\sigma$ by deleting one occurrence of $\nu^{-1}A_i$ for each
$1\le i\le r$.
\end{theorem}

\begin{proof}
We argue by induction on $r=|\mathfrak{m}_{c}|\ge1$. Let
$\{\Delta_{\imath}\}_{\imath\in I}$ and $\tau$ give a relevant
decomposition of $\mathcal{O}$, and put $s=|\mathfrak{m}_{c-1}|$.
Write $\Delta_1=A_1$ and use the indices
$i_1>\cdots>i_{k_1}>1$ from Lemma~\ref{lem:firstlast}(a).
First, all partners of the pieces of $A_1$ lie in the $\sigma$-part
of $\mathcal{O}$. Indeed, Lemma~\ref{lem:first} gives
$\bb(\Delta_{i_{k_1}})=\bb(A_1)-1$.
Since $A_1$ has the smallest beginning among the segments of $\mathfrak{m}_{c}$, we have $i_{k_1}>r$. Thus, every $i_j$ exceeds $r$, and every $\Delta_{i_j}$ has endpoint at most $c-1$. On the
other hand, Lemma~\ref{lem:first} gives $\ee(\Delta_{i_1})\ge\ee(\Delta_{1,1})-1=c-1$. 
Consequently, we have $\ee(\Delta_{i_1})=c-1$ and $i_1\le r+s$. As $r<i_{k_1}<\cdots<i_1$, all these partners belong to the $\sigma$-part and have endpoint $c-1$.
The first partner is an uncut segment. By
Lemma~\ref{lem:firstlast}(a), we have
$\Delta_{i_1,k_{i_1}}=\nu^{-1}\Delta_{1,1}$,
whose endpoint is $c-1$. Since $\Delta_{i_1}$ itself has endpoint
$c-1$, its bottom piece is the whole segment. Hence, $k_{i_1}=1$
and $\Delta_{i_1}=\nu^{-1}\Delta_{1,1}$.

We claim that $k_1=1$. If $k_1\ge2$, Lemma~\ref{lem:first} gives
$\bb(\Delta_{i_{k_1}})=\bb(A_1)-1$.
We have already shown that $\ee(\Delta_{i_{k_1}})=c-1$. These two
endpoints determine the whole partner segment:
$$
\Delta_{i_{k_1}}=\nu^{-1}A_1
\in\supp(\nu^{-1}\mathfrak{m}_{c}).
$$
But $i_{k_1-1}>i_{k_1}$, whereas Lemma~\ref{lem:first} gives
$\bb(\Delta_{i_{k_1-1}})>\bb(\Delta_{i_{k_1}})$. Thus, the
occurrence of $\nu^{-1}A_1$ at position $i_{k_1}$ is followed in
$\sigma$ by a segment with strictly larger beginning, contrary to
(A1). Therefore, $k_1=1$.

It follows that the blocks indexed by $1$ and $i_1$ each consist
of one piece, are paired by $\tau$, and contain $A_1$ and
$\nu^{-1}A_1$, respectively. In particular,
$\nu^{-1}A_1\in\mathfrak{m}_{c-1}$. Delete these two blocks and
restrict $\tau$ to the remaining indices. After order-preserving
relabelling, the restricted involution still satisfies (R1)--(R3). Hence, the order
$(A_2,\ldots,A_r,\ \sigma_1,\ \rho)$,
where $\sigma_1$ is obtained from $\sigma$ by deleting one
occurrence of $\nu^{-1}A_1$, admits a relevant decomposition.

By (A2), the resulting listing $\sigma_1$ is independent of which
occurrence of $\nu^{-1}A_1$ is deleted. It remains admissible for
$\mathfrak{m}'=\mathfrak{m}-A_1-\nu^{-1}A_1$: both (A1) and (A2)
are preserved by this deletion, and
$$
\supp\bigl(\nu^{-1}\mathfrak{m}'_{c}\bigr)
\subseteq\supp\bigl(\nu^{-1}\mathfrak{m}_{c}\bigr).
$$
If $r=1$, this proves both assertions. If $r>1$, then $c$ remains
the largest endpoint of $\mathfrak{m}'$, and the displayed order
has the form \eqref{eq:peelorder} for $\mathfrak{m}'$. Applying the
induction hypothesis proves
$$
\nu^{-1}\mathfrak{m}'_{c}\le\mathfrak{m}'_{c-1}
$$
and shows that $(\sigma',\rho)$ admits a relevant decomposition.
Adding back the paired occurrences gives
$\nu^{-1}\mathfrak{m}_{c}\le\mathfrak{m}_{c-1}$.
\end{proof}

\begin{corollary}\label{cor:firstlevel}
Let $\mathfrak{m}\ne0$ be distinguished, and let $c$ be its largest
endpoint. Then
$\nu^{-1}\mathfrak{m}_{c}\le\mathfrak{m}_{c-1}$,
or, equivalently, $S_{\mathfrak{m}}(\Delta)\ge0$ for every segment
$\Delta$ with $\ee(\Delta)=c-1$. Dually, if $a$ is the smallest beginning occurring in $\mathfrak{m}$, then
$\nu\mathfrak{m}^{a}\le\mathfrak{m}^{a+1}$.
\end{corollary}

\begin{proof}
Choose $\sigma$ by non-increasing beginnings and choose any
standard order $\rho$. The order \eqref{eq:peelorder} is standard
and, because $\mathfrak{m}$ is distinguished, admits a relevant
decomposition. The first inequality follows from
Theorem~\ref{thm:peel}. For $\ee(\Delta)=c-1$, all terms
$\mathfrak{m}(\nu^t\Delta)$ with $t\ge2$ vanish, so
$$
S_{\mathfrak{m}}(\Delta)
=\mathfrak{m}_{c-1}(\Delta)
-(\nu^{-1}\mathfrak{m}_{c})(\Delta),
$$
proving the stated equivalence.

For the dual assertion, $-a$ is the largest endpoint of
$\mathfrak{m}\dv$, and
$$
(\mathfrak{m}\dv)_{-a}=(\mathfrak{m}^{a})\dv,
\qquad
(\mathfrak{m}\dv)_{-a-1}=(\mathfrak{m}^{a+1})\dv.
$$
By Theorem~\ref{thm:duality}, $\mathfrak{m}\dv$ is distinguished.
Applying the first inequality to it and using \eqref{eq:nudual}
yields
$$
(\nu\mathfrak{m}^{a})\dv
=\nu^{-1}(\mathfrak{m}^{a})\dv
\le(\mathfrak{m}^{a+1})\dv.
$$
Taking duals proves the claim.
\end{proof}

\begin{remark}\label{rem:tension}
Theorem~\ref{thm:peel} produces a relevant decomposition after the
segments at endpoint $c$ and their $\nu^{-1}$-translates have been
removed. To apply the theorem again with $c-1$ as the largest
endpoint, the leftover layer
$R_{c-1}=\mathfrak{m}_{c-1}-\nu^{-1}\mathfrak{m}_{c}$
must appear in non-decreasing order of beginnings. This conflicts
with (A1) if a leftover occurrence of a segment in
$\supp(\nu^{-1}\mathfrak{m}_{c})$ has smaller beginning than
another segment in $R_{c-1}$: (A1) places the latter first, whereas
the new top layer would require the reverse order. Since
Corollary~\ref{cor:firstlevel} makes $c-1$ stable, the absence of
this obstruction is precisely the tameness condition at this first
peeling level. Further iteration also requires an admissible order
at each subsequent level.
The obstruction is reflected in the example of
Theorem~\ref{thm:E}: no standard order of $\mathfrak{m}_0$ with
non-increasing endpoints is a witness.
\end{remark}

\section{A numerical tameness condition}\label{sec:tame}

Recall from Definition~\ref{def:tameintro} that $x\in\Z$ is
\emph{stable} for $\mathfrak{m}$ if $S_{\mathfrak{m}}(\Gamma)\ge0$
for every segment $\Gamma$ with $\ee(\Gamma)\ge x$. The multiset
$\mathfrak{m}$ is \emph{tame} if, for every stable $x$ and every
segment $\Delta$ with $\ee(\Delta)=x$,
\begin{equation}\label{eq:tame}
S_{\mathfrak{m}}(\Delta)>0
\ \text{and}\
S_{\mathfrak{m}}(\nu\Delta)>0
\quad\Longrightarrow\quad
\bb(\Delta)=
\max\{\bb(\Gamma):
\ee(\Gamma)=x,\ S_{\mathfrak{m}}(\Gamma)>0\}.
\end{equation}
A segment is determined by its beginning and end. Consequently,
at most one segment with a given stable end can satisfy the
antecedent of \eqref{eq:tame}.

In the notation of Remark~\ref{rem:peelingnumbers}, $x$ is stable
if and only if $R_y\ge0$ for every $y\ge x$. When $x$ is stable,
condition \eqref{eq:tame} says that any segment belonging to both
$\supp R_x$ and $\supp(\nu^{-1}R_{x+1})$ has the largest beginning
among the segments of $\supp R_x$.

\begin{theorem}\label{thm:main}
Let $\mathfrak{m}\in\Oz$ be tame. If $\mathfrak{m}$ is
distinguished, then it is of Speh type. The same conclusion holds
if $\mathfrak{m}\dv$ is tame.
\end{theorem}

\begin{proof}
By Theorem~\ref{thm:duality} and Lemma~\ref{lem:spehdual}, the
assertion for a multiset whose dual is tame follows from the
first assertion. We therefore assume that $\mathfrak{m}$ is tame
and distinguished. The result is immediate if $\mathfrak{m}=0$,
so suppose that $\mathfrak{m}\ne0$, and let $c$ be its largest
endpoint.

Suppose, for a contradiction, that $\mathfrak{m}$ is not of Speh
type. For each $x\le c$, let $R_x$ be the possibly signed
multiset of Remark~\ref{rem:peelingnumbers}; thus
$R_x(\Delta)=S_{\mathfrak{m}}(\Delta)$ when $\ee(\Delta)=x$.
By Theorem~\ref{thm:speh}, some $R_x$ takes a negative value.
Choose $x^{*}$ to be the largest such $x$. Since
$R_c=\mathfrak{m}_c\ge0$, we have $x^{*}\le c-1$ and
\begin{equation}\label{eq:Rpos}
R_y\ge0\quad\text{for every }y>x^{*},
\qquad \text{while}~
R_{x^{*}}\not\ge0.
\end{equation}

\emph{Construction of a standard order.}
For $x^{*}<y\le c-1$, put $N_y=\nu^{-1}R_{y+1}$ and $L_y=R_y=\mathfrak{m}_y-N_y$. 
By \eqref{eq:Rpos}, both $N_y$ and $L_y$ are multisets of
segments with end $y$. We construct a listing $\sigma_y$ of
$\mathfrak{m}_y$.
If $L_y=0$, list $\mathfrak{m}_y$ by non-increasing beginnings. Otherwise, set
$$
\beta_y=\max\{\bb(\Gamma):\Gamma\in\supp L_y\},
$$
and let $X_y$ be the segment with end $y$ and beginning $\beta_y$.
Define the following sub-multisets, with occurrences counted
with multiplicity:
$$
\begin{aligned}
P_1&=\{\Gamma\in\mathfrak{m}_y:
             \bb(\Gamma)>\beta_y\},\\
P_2&=\{\Gamma\in L_y:
             \bb(\Gamma)<\beta_y\},\\
P_3&=\mathfrak{m}_y-P_1-P_2.
\end{aligned}
$$
Let $\sigma_y$ list $P_1$ by non-increasing beginnings, followed
by $P_2$ by non-decreasing beginnings, and finally $P_3$ by
non-increasing beginnings. For $y\le x^{*}$, list
$\mathfrak{m}_y$ by non-increasing beginnings. List
$\mathfrak{m}_c$ by non-decreasing beginnings, and denote this
listing by $\sigma_c$. Concatenating the nonempty endpoint layers gives
$\mathcal{O}
=(\sigma_c,\sigma_{c-1},\sigma_{c-2},\ldots)$.
Only finitely many entries occur in this expression. Its
endpoints are non-increasing, so $\mathcal{O}$ is standard by
Lemma~\ref{lem:orders}(a).

\emph{Compatibility claim.}
For every $x^{*}<y\le c-1$, the listing $\sigma_y$ has the
following properties:

(i) $\sigma_y$ is admissible for the multiset
$R_{y+1}+\mathfrak{m}_y+\mathfrak{m}_{y-1}+\cdots$.

(ii) Deleting from $\sigma_y$ one occurrence of $\nu^{-1}A$
for each occurrence $A$ in $R_{y+1}$ leaves the listing of
$L_y$ by non-decreasing beginnings.

If $L_y=0$, then $\mathfrak{m}_y=N_y$. The listing by
non-increasing beginnings is admissible, and the deletion leaves
the empty listing. Suppose henceforth that $L_y\ne0$.
Because $L_y$ has no segment with beginning greater than
$\beta_y$, we have $P_1\le N_y$. Moreover, $y$ is stable by
\eqref{eq:Rpos}. If
$\Delta\in\supp L_y\cap\supp N_y$, then
$$
S_{\mathfrak{m}}(\Delta)=L_y(\Delta)>0,
\qquad
S_{\mathfrak{m}}(\nu\Delta)
=R_{y+1}(\nu\Delta)>0.
$$
Therefore, tameness gives $\bb(\Delta)=\beta_y$, and hence $\Delta=X_y$. In particular, every segment occurring in $P_2$
lies outside $\supp N_y$. Every occurrence of a segment in
$\supp N_y$ with beginning at most $\beta_y$ consequently lies
in $P_3$. No segment has occurrences split between two of $P_1,P_2,P_3$.
The beginning separates $P_1$ from the other two parts. If
$\bb(\Gamma)<\beta_y$, then either
$\Gamma\notin\supp N_y$, in which case all its occurrences lie
in $P_2$, or $\Gamma\in\supp N_y$, in which case
$L_y(\Gamma)=0$ and all its occurrences lie in $P_3$.
Occurrences of $X_y$ all lie in $P_3$. Each part is listed
monotonically by beginning, so equal segments are consecutive.
This proves (A2).

To verify (A1), let $\Gamma\in\supp N_y$. If
$\bb(\Gamma)>\beta_y$, all its occurrences lie in $P_1$ and
follow every segment with larger beginning. If
$\bb(\Gamma)\le\beta_y$, all its occurrences lie in $P_3$.
Segments with larger beginning lie either in the earlier parts
$P_1$ and $P_2$, or earlier in the non-increasing listing of
$P_3$. Thus, $\sigma_y$ is admissible, proving (i).

Finally, the deletion removes all of $P_1$ and all of $P_3$
except its $L_y(X_y)$ occurrences of $X_y$; it leaves $P_2$
untouched. The remaining listing consists of $P_2$ by
non-decreasing beginnings, followed by $X_y$, whose beginning
is the largest in $L_y$. This proves (ii). Notice also that if
$R_{y+1}=0$, then $N_y=0$ and the construction gives
$\sigma_y$ itself as the listing of $R_y$ by non-decreasing
beginnings.

\emph{Peeling.}
For $c\ge y>x^{*}$, let $\lambda_y$ be the listing of $R_y$ by
non-decreasing beginnings, including the empty listing when
$R_y=0$. We show, by descending induction on $y$, that
\begin{equation}\label{eq:step}
\mathcal{O}_y
=(\lambda_y,\sigma_{y-1},\sigma_{y-2},\ldots)
\quad\text{admits a relevant decomposition}
\qquad \text{for}~ c\ge y>x^{*}.
\end{equation}
The order $\mathcal{O}_y$ is an order of
$R_y+\mathfrak{m}_{y-1}+\mathfrak{m}_{y-2}+\cdots$ and is
standard because its endpoint layers are arranged in
non-increasing order.

For $y=c$, we have $\lambda_c=\sigma_c$ and
$\mathcal{O}_c=\mathcal{O}$. Since $\mathfrak{m}$ is
distinguished, $\mathcal{O}_c$ admits a relevant decomposition.
Suppose that $y>x^{*}+1$ and that $\mathcal{O}_y$ admits a relevant decomposition. If $R_y=0$, then
$N_{y-1}=\nu^{-1}R_y=0$. By the observation following the
compatibility claim, $\sigma_{y-1}=\lambda_{y-1}$. Hence, we have
$\mathcal{O}_y=\mathcal{O}_{y-1}$, and the relevant
decomposition persists.
If $R_y\ne0$, then $y$ is the largest endpoint of the
multiset ordered by $\mathcal{O}_y$. Its top layer $R_y$ is
listed by non-decreasing beginnings, and its next layer
$\mathfrak{m}_{y-1}$ is listed by the admissible
$\sigma_{y-1}$. Theorem~\ref{thm:peel} applies. It gives a
relevant decomposition after deleting $R_y$ and one occurrence
of $\nu^{-1}A$ for every occurrence $A$ in $R_y$. By part (ii)
of the compatibility claim, the resulting order is
$\mathcal{O}_{y-1}$. This completes the descending induction.

Now $R_{x^{*}+1}\ne0$. Indeed, the recursion
$
R_{x^{*}}
=\mathfrak{m}_{x^{*}}-\nu^{-1}R_{x^{*}+1}
$
and $R_{x^{*}}\not\ge0$ imply that
$\nu^{-1}R_{x^{*}+1}\ne0$. Thus, $\mathcal{O}_{x^{*}+1}$
has a nonempty top layer and, by \eqref{eq:step}, admits a
relevant decomposition. Its next layer $\sigma_{x^{*}}$ is
listed by non-increasing beginnings and is therefore
admissible. Theorem~\ref{thm:peel} yields
$\nu^{-1}R_{x^{*}+1}\le\mathfrak{m}_{x^{*}}$,
or equivalently $R_{x^{*}}\ge0$. This contradicts
\eqref{eq:Rpos}. Hence, $\mathfrak{m}$ is of Speh type.
\end{proof}

\begin{corollary}\label{cor:classes}
Hypotheses~\ref{hyp:85} and~\ref{hyp:86} hold for
$\mathfrak{m}$ in each of the following cases:
\begin{enumerate}[label=(\alph*)]
\item $\mathfrak{m}$ is a set;
\item $|\mathfrak{m}_{x}|\le2$ for every $x\in\Z$
      (equivalently, $|\mathfrak{m}[i]|\le2$ for every $i$);
\item $|\mathfrak{m}^{x}|\le2$ for every $x\in\Z$;
\item more generally, $\mathfrak{m}$ or $\mathfrak{m}\dv$
      is tame.
\end{enumerate}
\end{corollary}

\begin{proof}
By Theorem~\ref{thm:main} and
Corollary~\ref{cor:85iff86}, it suffices to show that the
multisets in (a) and (b) are tame. Part (c) then follows
from (b) by duality, since
$|\mathfrak{m}^{x}|=|(\mathfrak{m}\dv)_{-x}|$.
Let $x$ be stable, and suppose that a segment $\Delta$ with $\ee(\Delta)=x$ satisfies the antecedent of \eqref{eq:tame}. The recurrence \eqref{eq:Srec} gives
$\mathfrak{m}(\Delta)
=S_{\mathfrak{m}}(\Delta)
 +S_{\mathfrak{m}}(\nu\Delta)\ge2$.
Moreover, stability gives $R_{x+1}\ge0$, so the peeling
recursion yields
$R_x=\mathfrak{m}_x-\nu^{-1}R_{x+1}\le\mathfrak{m}_x$.
Consequently, every segment $\Gamma$ with
$\ee(\Gamma)=x$ and $S_{\mathfrak{m}}(\Gamma)>0$
belongs to $\supp\mathfrak{m}_x$.

If $\mathfrak{m}$ is a set, then $\mathfrak{m}(\Delta)\le1$, a contradiction. Thus, the
antecedent of \eqref{eq:tame} never occurs, proving (a).
If $|\mathfrak{m}_x|\le2$, then
$\mathfrak{m}(\Delta)\ge2$ forces
$\mathfrak{m}_x=\{\Delta,\Delta\}$. Hence, $\Delta$ is
the only segment with end $x$ for which
$S_{\mathfrak{m}}$ can be positive. In
particular, it has the largest beginning among such segments,
as required by \eqref{eq:tame}. This proves (b).
\end{proof}
\begin{example}\label{ex:strict} Let $\mathfrak{m}=\{[0,1],[1,2],[1,2],[1,2],[2,3]\}$. Its nonzero endpoint layers are 
	$$\mathfrak{m}_3=\{[2,3]\},\qquad \mathfrak{m}_2=3\{[1,2]\},\qquad \mathfrak{m}_1=\{[0,1]\}.$$
	 Using the peeling recursion of Remark~\ref{rem:peelingnumbers}, we obtain $R_3=\{[2,3]\}$. So, $R_2=\mathfrak{m}_2-\nu^{-1}R_3=3\{[1,2]\}-\{[1,2]\}=2\{[1,2]\}$, while $R_1=\mathfrak{m}_1-\nu^{-1}R_2=\{[0,1]\}-2\{[0,1]\}=-\{[0,1]\}.$ Thus, the stable integers are precisely those $x\geq 2$. This multisegment is not covered by any of the cases of Corollary~\ref{cor:classes}(a)--(c). Indeed, $\mathfrak{m}$ is not a set, and $|\mathfrak{m}^{1}|= |\mathfrak{m}_2|=3$. Moreover, the two extremal inequalities of Corollary~\ref{cor:firstlevel} are satisfied. At the largest endpoint, \(\nu^{-1}\mathfrak{m}_3=\{[1,2]\}\leq 3\{[1,2]\}=\mathfrak{m}_2.\) On the other hand, $0$ is the smallest beginning occurring in $\mathfrak{m}$, and \(\nu\mathfrak{m}^{0}=\{[1,2]\}\leq 3\{[1,2]\}=\mathfrak{m}^{1}.\) Thus, the failure of distinction cannot be detected by the first-order extremal constraints of Corollary~\ref{cor:firstlevel}. We next verify that $\mathfrak{m}$ is tame. At the stable level $x=2$, one has \(S_{\mathfrak{m}}([1,2])=\mathfrak{m}([1,2])-\mathfrak{m}([2,3])=3-1=2>0,\) and \(S_{\mathfrak{m}}(\nu[1,2])=S_{\mathfrak{m}}([2,3])=1>0.\) Furthermore, $[1,2]$ is the only segment $\Gamma$ with $e(\Gamma)=2$ for which $S_{\mathfrak{m}}(\Gamma)>0$. Consequently, \(b([1,2])=\max\{b(\Gamma):e(\Gamma)=2,\;S_{\mathfrak{m}}(\Gamma)>0\},\) so the tameness condition is satisfied at $x=2$. At the stable level $x=3$, although \(S_{\mathfrak{m}}([2,3])=1,\) we have \(S_{\mathfrak{m}}(\nu[2,3])=S_{\mathfrak{m}}([3,4])=0,\) so the antecedent in the tameness condition does not occur. For $x>3$, all relevant $S_{\mathfrak{m}}$-values vanish. Hence, $\mathfrak{m}$ is tame. On the other hand, \(S_{\mathfrak{m}}([0,1])=\mathfrak{m}([0,1])-\mathfrak{m}([1,2])+\mathfrak{m}([2,3])=1-3+1=-1<0.\) Therefore, $\mathfrak{m}$ is not of Speh type by Theorem~\ref{thm:speh}. Since $\mathfrak{m}$ is tame, Theorem~\ref{thm:main} implies that $\mathfrak{m}$ is not distinguished. Thus, this example lies strictly beyond the cases covered by Corollary~\ref{cor:classes}(a)--(c), satisfies both first-order extremal constraints of Corollary~\ref{cor:firstlevel}, and provides a genuinely non-vacuous application of the tameness condition. \end{example}
\begin{remark}\label{rem:tamecheck}
Tameness depends only on $\mathfrak{m}$ and can be
checked by computing the values of $S_{\mathfrak{m}}$
from the largest endpoint downward, using the recursion
of Remark~\ref{rem:peelingnumbers}. Once a negative
level is reached, no lower integer is stable.
The condition is not vacuous, and four is the minimum
number of segment occurrences in a non-tame multiset.
Indeed, failure of \eqref{eq:tame} at a stable integer
$x$ requires a segment $\Delta$ with
$S_{\mathfrak{m}}(\Delta)>0$ and
$S_{\mathfrak{m}}(\nu\Delta)>0$, together with a
different segment $\Gamma$ of end $x$ such that
$S_{\mathfrak{m}}(\Gamma)>0$ and
$\bb(\Gamma)>\bb(\Delta)$. The recurrence
\eqref{eq:Srec} gives
$\mathfrak{m}(\Delta)\ge2$, while stability gives
$\mathfrak{m}(\Gamma)\ge1$. Thus, at least three
occurrences have end $x$. The positive value
$S_{\mathfrak{m}}(\nu\Delta)$ also requires at
least one occurrence with end greater than $x$,
giving at least four in total. This bound is attained by $\mathfrak{m}=\{[0,1],[0,1],[1,1],[1,2]\}$.
Here,
$$
R_2=\{[1,2]\},
\qquad
\nu^{-1}R_2=\{[0,1]\},
\qquad
R_1=\{[0,1],[1,1]\}.
$$
The integer $1$ is stable, but $[0,1]$ belongs to
$\supp(\nu^{-1}R_2)\cap\supp R_1$ and has smaller
beginning than $[1,1]\in\supp R_1$. Hence, this
multiset is not tame.
\end{remark}

\section{Failure of the endpoint-order strategy}\label{sec:counter}

The proofs of \cite[Propositions 8.7 and 8.8]{MOS} use the canonical
standard order of \cite[\S8.0.12]{MOS}, whose endpoints are
non-increasing. Likewise, for a tame multiset that is not of Speh
type, the order constructed in the proof of Theorem~\ref{thm:main}
has non-increasing endpoints and is a strong witness. We now show that, in general, no standard order with non-increasing endpoints need be a witness. 

\begin{theorem}\label{thm:counter}
Let
$\mathfrak{m}_{0}=\{[0,0],[0,1],[0,1],[1,1],[1,2]\}$.
Then:
\begin{enumerate}[label=(\alph*)]
\item $\mathfrak{m}_{0}$ is not of Speh type;

\item $\mathfrak{m}_{0}$ has exactly nine standard orders. Six of them are
$$
\begin{aligned}
\mathcal{G}_{1}
 &=([1,1],[1,2],[0,1],[0,0],[0,1]),
&\quad
\mathcal{G}_{2}
 &=([1,1],[1,2],[0,1],[0,1],[0,0]),\\
\mathcal{G}_{3}
 &=([1,2],[0,1],[0,1],[1,1],[0,0]),
&
\mathcal{G}_{4}
 &=([1,2],[0,1],[1,1],[0,1],[0,0]),\\
\mathcal{G}_{5}
 &=([1,2],[1,1],[0,1],[0,0],[0,1]),
&
\mathcal{G}_{6}
 &=([1,2],[1,1],[0,1],[0,1],[0,0]).
\end{aligned}
$$
The remaining three are 
$$
\begin{aligned}
\mathcal{W}_{1}&=([1,2],[1,1],[0,0],[0,1],[0,1]),\\
\mathcal{W}_{2}&=([1,1],[1,2],[0,0],[0,1],[0,1]),\\
\mathcal{W}_{3}&=([1,2],[0,1],[1,1],[0,0],[0,1]);
\end{aligned}
$$

\item each of $\mathcal{G}_{1},\ldots,\mathcal{G}_{6}$
admits a non-trivial relevant decomposition, whereas none of $\mathcal{W}_{1}$, $\mathcal{W}_{2}$, $\mathcal{W}_{3}$
admits any relevant decomposition. Consequently,
$\mathfrak{m}_{0}$ is not distinguished, and the three
$\mathcal{W}$-orders are precisely its strong witnesses;

\item the standard orders of $\mathfrak{m}_{0}$ with
non-increasing endpoints are exactly
$\mathcal{G}_{3},\mathcal{G}_{4},\mathcal{G}_{6}$.
Hence, no witness for $\mathfrak{m}_{0}$ has
non-increasing endpoints. On the other hand,
$\mathcal{W}_{1}$ and $\mathcal{W}_{2}$ have
non-increasing beginnings.
\end{enumerate}
Thus, the witness strategy of \cite[\S8.0.10]{MOS},
when restricted to standard orders with non-increasing
endpoints, cannot be applied to all multisets.
\end{theorem}
\begin{proof}
\emph{(a).}
Among segments of size two, the only nonzero
multiplicities are $\mathfrak{m}_{0}([0,1])=2$ and
$\mathfrak{m}_{0}([1,2])=1$. It follows that
$S_{\mathfrak{m}_{0}}([-1,0])=0-2+1=-1<0$.
By Theorem~\ref{thm:speh}, $\mathfrak{m}_{0}$ is
not of Speh type.

\emph{(b).}
Write $A=[1,2]$, $B=[1,1]$, $C=[0,0]$, and $D=[0,1]$,
where $D$ occurs twice. The only precedences among
these segments are $C\prec B$, $C\prec A$, and $D\prec A$. Thus, a listing is standard if and only if $A$ precedes $C$ and both occurrences of $D$, and $B$ precedes $C$.
The entry $A$ must occupy position $1$ or $2$.
If $A$ is first, there are $4!/2!=12$ listings of
$B,C,D,D$, six of which place $B$ before $C$.
If $A$ is second, $B$ must be first, and the
remaining entries $C,D,D$ have $3!/2!=3$ listings.
This gives $6+3=9$ standard orders, precisely those
in the statement.

\emph{(c): the orders $\mathcal{G}_{i}$.}
For each order in the following table, decompose the
indicated block and leave all other blocks undecomposed.
The third column gives the two-element orbits of the
involution $\tau$. In the last column, the two lines
of each entry form one (R1) inequality for
$\varphi=\pr_{1}\circ\tau$.

$$
\resizebox{\linewidth}{!}{$\displaystyle
\renewcommand{\arraystretch}{1.25}
\begin{array}{llll}
\text{Order} & \hspace{-.1cm} \text{Blocks decomposed} & \hspace{.5cm} \tau & \text{(R1)}\\ \hline
\hspace{.2cm} \mathcal{G}_{1} & \hspace{-.2cm} \Delta_{3}=[0,1]=([1,1],[0,0]) &
 \{(1,1),(3,2)\},\{(2,1),(5,1)\},\{(3,1),(4,1)\} & \varphi(3,1)=4>1=\varphi(3,2)\\
\hspace{.2cm} \mathcal{G}_{2} & \hspace{-.2cm} \Delta_{4}=[0,1]=([1,1],[0,0]) &
 \{(1,1),(4,2)\},\{(2,1),(3,1)\},\{(4,1),(5,1)\} & \varphi(4,1)=5>1=\varphi(4,2)\\
\hspace{.2cm} \mathcal{G}_{4} & \hspace{-.2cm} \Delta_{4}=[0,1]=([1,1],[0,0]) &
 \{(1,1),(2,1)\},\{(3,1),(4,2)\},\{(4,1),(5,1)\} & \varphi(4,1)=5>3=\varphi(4,2)\\
\hspace{.2cm} \mathcal{G}_{5} & \hspace{-.2cm} \Delta_{3}=[0,1]=([1,1],[0,0]) &
 \{(1,1),(5,1)\},\{(2,1),(3,2)\},\{(3,1),(4,1)\} & \varphi(3,1)=4>2=\varphi(3,2)\\
\hspace{.2cm} \mathcal{G}_{6} & \hspace{-.2cm} \Delta_{4}=[0,1]=([1,1],[0,0]) &
 \{(1,1),(3,1)\},\{(2,1),(4,2)\},\{(4,1),(5,1)\} & \varphi(4,1)=5>2=\varphi(4,2)
\end{array}
$}
$$

In every row, $\tau$ is fixed-point-free, giving (R2).
Condition (R3) holds orbit by orbit because the earlier
piece is the $\nu$-translate of its partner; the relevant
identities are $[1,2]=\nu[0,1]$ and
$[1,1]=\nu[0,0]$. The final column verifies the
only instance of (R1). Hence, these five decompositions
are relevant and non-trivial.

For
$\mathcal{G}_{3}=([1,2],[0,1],[0,1],[1,1],[0,0])$,
decompose the first three blocks as
$$
\Delta_{1}=[1,2]=([2,2],[1,1]),\qquad
\Delta_{2}=[0,1]=([1,1],[0,0]),\qquad
\Delta_{3}=[0,1]=([1,1],[0,0]).
$$
Leave blocks $4$ and $5$ undecomposed and take
$$
\begin{aligned}
\tau:\quad
&\{(1,1),(4,1)\},\qquad \{(1,2),(2,2)\},\\
&\{(2,1),(3,2)\},\qquad \{(3,1),(5,1)\}.
\end{aligned}
$$
Each orbit satisfies (R3), while (R1) follows from
$$
\begin{aligned}
\varphi(1,1)=4&>2=\varphi(1,2),\\
\varphi(2,1)=3&>1=\varphi(2,2),\\
\varphi(3,1)=5&>2=\varphi(3,2).
\end{aligned}
$$
This gives the sixth non-trivial relevant decomposition.

\emph{The order $\mathcal{W}_{1}$.}
Suppose that $\mathcal{W}_{1}$ admitted a relevant
decomposition. Its last block is $\Delta_{5}=[0,1]$.
By Lemma~\ref{lem:last}, the partner of its top
piece belongs to a block with end $2$. The only
such block is $\Delta_{1}=[1,2]$, so $i'_{1}=1$.
If $k_{5}\ge2$, the same lemma would give
$i'_{2}<1$, which is impossible. Thus, $k_{5}=1$,
and $\Delta_{1,1}=\nu[0,1]=[1,2]$ forces
$k_{1}=1$ and $\tau(1,1)=(5,1)$.
Delete these paired one-piece blocks. Restricting
$\tau$ and relabelling the remaining blocks in order
preserves (R1)--(R3), and would give a relevant
decomposition of $([1,1],[0,0],[0,1])$. Applied
to its last block $[0,1]$, Lemma~\ref{lem:last}
would require an earlier block with end $2$.
Neither remaining block has that end.

\emph{The order $\mathcal{W}_{3}$.}
The same argument applied to its last block
$\Delta_{5}=[0,1]$ forces
$k_{5}=k_{1}=1$ and $\tau(1,1)=(5,1)$.
Deletion would therefore give a relevant
decomposition of $([0,1],[1,1],[0,0])$.
Its last block $[0,0]$ is undecomposed, and
Lemma~\ref{lem:last} requires an earlier top
piece equal to $\nu[0,0]=[1,1]$.
If that piece is the block $[1,1]$, this block
is paired with $[0,0]$ and all pieces of the
remaining block $[0,1]$ would have to be
matched within their own block, contrary to
Lemma~\ref{lem:blocks}. If it is the top piece
of $[0,1]$, then
$[0,1]=([1,1],[0,0])$. Its bottom piece must
be paired with the block $[1,1]$, but (R3)
would require $[0,0]=\nu[1,1]=[2,2]$, a contradiction.

\emph{The order $\mathcal{W}_{2}$.}
Here, the last block is again $\Delta_{5}=[0,1]$,
but the only block with end $2$ is
$\Delta_{2}=[1,2]$. Lemma~\ref{lem:last}
therefore gives $i'_{1}=2$ and
$\Delta_{2,1}=\nu\Delta_{5,1}$.

If $k_{5}=1$, then
$\Delta_{2,1}=\nu[0,1]=[1,2]=\Delta_{2}$.
Hence, $k_{2}=1$, the two blocks are paired,
and their deletion would give a relevant
decomposition of $([1,1],[0,0],[0,1])$.
This was ruled out in the case of $\mathcal{W}_{1}$.
If $k_{5}=2$, then
$\Delta_{5}=([1,1],[0,0])$ and
$\Delta_{2}=[1,2]=([2,2],[1,1])$, with
$\tau(2,1)=(5,1)$. Lemma~\ref{lem:last}
further gives $i'_{2}<2$, so $i'_{2}=1$
and $\tau(1,1)=(5,2)$. The unmatched
indices are $(2,2)$, block $3=[0,0]$,
and the pieces of block $4=[0,1]$.
Lemma~\ref{lem:blocks} forces $(2,2)$
to be paired with block $3$ or block $4$.
If it is paired with block $3$, all pieces
of block $4$ would have to be matched
within block $4$, contradicting
Lemma~\ref{lem:blocks}. If it is paired
with block $4$, (R3) requires its partner
to be $[0,0]$. Thus,
$\Delta_{4}=([1,1],[0,0])$ and
$\tau(2,2)=(4,2)$. The remaining indices
$(3,1)$ and $(4,1)$ must be paired, but
(R3) would then give $[0,0]=\nu[1,1]=[2,2]$.
Thus, none of the three $\mathcal{W}$-orders
admits a relevant decomposition.

\emph{(d).}
An order with non-increasing endpoints must
begin with $[1,2]$ and end with $[0,0]$.
The three middle entries are an arbitrary
listing of $[0,1],[0,1],[1,1]$, giving
exactly $\mathcal{G}_{3},\mathcal{G}_{4},
\mathcal{G}_{6}$. Each admits a relevant
decomposition by (c), so none is a witness.
Finally, the beginnings in both
$\mathcal{W}_{1}$ and $\mathcal{W}_{2}$
are $1,1,0,0,0$.
\end{proof}
\begin{remark}\label{rem:countermeaning}
	Although $\mathfrak m_0$ is not tame, its dual is tame.
	Hence, its non-distinguishedness is already implied by Theorem~\ref{thm:main}
	together with Theorem~\ref{thm:duality}. The role of Theorem~\ref{thm:counter} is therefore not to produce a new obstruction to the hypotheses, but to show that the
	endpoint-order strategy of \cite[\S8.0.10]{MOS} cannot work in full
	generality.
\end{remark}

\begin{remark}\label{rem:minimality}
The five-segment multiset in Theorem~\ref{thm:counter}
has a simple relation to the smallest non-tame example
of Remark~\ref{rem:tamecheck}. Namely, if
$\mathfrak{n}=\{[0,1],[0,1],[1,1],[1,2]\}$, then
$\mathfrak{m}_{0}=\mathfrak{n}+\{[0,0]\}$.
Thus, adjoining a single segment to a four-segment
non-tame multiset yields the endpoint-order
obstruction of Theorem~\ref{thm:counter}.
\end{remark}

\begin{remark}
Let $\mathcal{C}(\mathfrak{m})$ be the canonical
standard order of \cite[\S8.0.12]{MOS}, and put
$\mathcal{C}^{*}(\mathfrak{m})
=(\mathcal{C}(\mathfrak{m}\dv))\dv$.
Theorem~\ref{thm:duality} shows that both are
standard orders of $\mathfrak{m}$ and that
dualizing preserves the existence and triviality
of relevant decompositions.

Theorem~\ref{thm:counter} illustrates the reason
to consider the second order: the strong witnesses
$\mathcal{W}_{1}$ and $\mathcal{W}_{2}$ have
non-increasing beginnings, and hence their dual
orders have non-increasing endpoints. This leads
to the following question.
\end{remark}

\begin{question}\label{q:twoorders}
Is it true that, for every $\mathfrak{m}\in\Oz$,
at least one of $\mathcal{C}(\mathfrak{m})$ and
$\mathcal{C}^{*}(\mathfrak{m})$ admits no
non-trivial relevant decomposition?
\end{question}

An affirmative answer would imply
Hypothesis~\ref{hyp:85}. Indeed, if
$\mathfrak{m}$ is distinguished, whichever
order is specified by the question admits a
relevant decomposition. As it admits no
non-trivial one, the decomposition is trivial;
Lemma~\ref{lem:trivialrel} then implies that
$\mathfrak{m}$ is of Speh type.
Hypothesis~\ref{hyp:86} would follow from
Corollary~\ref{cor:85iff86}.

\section{Representation-theoretic consequences and open problems}\label{sec:app}

\subsection{Relation with symplectic periods}

Let $\mathrm{F}$ be a non-archimedean local field, and let $\rho$ be
an irreducible supercuspidal representation of
$\mathrm{GL}_{d}(\mathrm{F})$. Denote by $\mathcal{O}_{\rho}$
the set of multisets of segments in the line
$\rho^{\Z}$. The $\rho$-labelling
$[a,b]\mapsto[a,b]_{(\rho)}$ identifies $\Oz$ with
$\mathcal{O}_{\rho}$.

The proof of \cite[Proposition 9.1]{MOS} shows that,
if the standard module $\lambda(\mathfrak{M})$
admits a symplectic period, then the unlabelling of
$\mathfrak{M}$ is distinguished in the sense of
Definition~\ref{def:distspeh}. Our combinatorial
results therefore give necessary conditions for
such a period.

\begin{corollary}\label{cor:rep}
Let $0\ne\mathfrak{M}\in\mathcal{O}_{\rho}$, and
suppose that the standard module
$\lambda(\mathfrak{M})$ admits a symplectic period.
Let $\mathfrak{m}\in\Oz$ be the unlabelling of
$\mathfrak{M}$, let $c$ be its largest endpoint,
and let $a$ be its smallest beginning. Then $\nu^{-1}\mathfrak{m}_{c}\le\mathfrak{m}_{c-1}$ and $\nu\mathfrak{m}^{a}\le\mathfrak{m}^{a+1}$.  Moreover, 
if $\mathfrak{m}$ or
$\mathfrak{m}\dv$ is tame, then
$\mathfrak{m}=\mathfrak{n}+\nu\mathfrak{n}$
for some $\mathfrak{n}\in\Oz$; that is,
$\mathfrak{m}$ is of Speh type. This conclusion
applies, in particular, when $\mathfrak{m}$ is
a set, when at most two of its segments share
any given endpoint, or when at most two share
any given beginning, with occurrences counted
with multiplicity.
\end{corollary}

\begin{proof}
As recalled above, the existence of a symplectic
period on $\lambda(\mathfrak{M})$ implies that
$\mathfrak{m}$ is distinguished. The two
inequalities follow from
Corollary~\ref{cor:firstlevel}. If
$\mathfrak{m}$ or $\mathfrak{m}\dv$ is tame,
Theorem~\ref{thm:main} implies that
$\mathfrak{m}$ is of Speh type. The three
particular cases follow from
Corollary~\ref{cor:classes}.
\end{proof}

The inequalities give unconditional numerical
constraints on the unlabelled multiset:
every occurrence in its top endpoint layer
requires a matching occurrence one step below,
and the analogous condition holds at the
smallest beginning. The Speh-type conclusion
extends the range of the corresponding result
of \cite[Corollary 9.2]{MOS}.

\subsection{Beyond the tame case}

Our results do not settle Hypothesis~\ref{hyp:85} for
arbitrary multisets, particularly when neither
$\mathfrak{m}$ nor $\mathfrak{m}\dv$ is tame.
Corollary~\ref{cor:85iff86} shows that passing to
Hypothesis~\ref{hyp:86} does not change the
combinatorial problem. Theorem~\ref{thm:counter}
also shows that a proof based on finding a witness
cannot restrict its search to standard orders with
non-increasing endpoints. Two possible approaches
are the following.

\begin{enumerate}[label=(\roman*),leftmargin=2.1em]
\item Answer Question~\ref{q:twoorders} affirmatively.
Theorem~\ref{thm:duality} makes the canonical order
of $\mathfrak{m}$ and the reversed dual of the canonical order
of $\mathfrak{m}\dv$ a natural pair. A difficulty
is to control relevant decompositions in both orders
within a single peeling argument.

\item Establish the following reduction step:
if $\mathfrak{m}$ is distinguished, then
$\mathfrak{m}' =\mathfrak{m}-\mathfrak{m}_{c} -\nu^{-1}\mathfrak{m}_{c}$
is distinguished, where $c$ is the largest
endpoint of $\mathfrak{m}$. By
Corollary~\ref{cor:firstlevel}, $\mathfrak{m}'$
is a multiset, and $|\mathfrak{m}'|<|\mathfrak{m}|$.
The reduction step would therefore prove
Hypothesis~\ref{hyp:85} by induction: if
$\mathfrak{m}'=\mathfrak{n}+\nu\mathfrak{n}$,
then $\mathfrak{m}
=(\mathfrak{n}+\nu^{-1}\mathfrak{m}_{c})
 +\nu(\mathfrak{n}+\nu^{-1}\mathfrak{m}_{c})$,
so $\mathfrak{m}$ is also of Speh type.
Every standard order of $\mathfrak{m}'$ can be
extended to a standard order of $\mathfrak{m}$
by placing the segments of $\mathfrak{m}_{c}$
and then those of $\nu^{-1}\mathfrak{m}_{c}$
before it. The obstruction is different:
a relevant decomposition of the extended order
need not pair the added segments with one another.
Without such pairings, its involution cannot
simply be restricted to the order of
$\mathfrak{m}'$.
\end{enumerate}

Corollary~\ref{cor:sameend} provides additional
necessary local conditions. For each segment
of maximal end, a distinguished multiset
contains a segment whose beginning is one
smaller; for each segment of minimal beginning,
it contains a segment whose end is one larger.
These conditions alone are not sufficient:
$\mathfrak{m}_{0}$ satisfies both, but is not
distinguished by Theorem~\ref{thm:counter}.
Understanding how such local constraints behave
under the proposed reduction may help clarify
the remaining obstruction.

\subsection{Converse questions}\label{ss:converses}

It is useful to distinguish three converse statements.

\emph{(1) The combinatorial converse.}
Every multiset of Speh type is distinguished by Lemma~\ref{lem:trivialrel}. Thus, the content of Hypothesis~\ref{hyp:85} is the reverse implication. By Corollary~\ref{cor:85iff86}, Hypothesis~\ref{hyp:86} has the same content. The resulting characterization,
``distinguished if and only if of Speh type,'' holds, in particular, when the multiset or its dual is tame, by Theorem~\ref{thm:main}.

\emph{(2) The order-level assertion.}
Speh type does not imply that the trivial decomposition is the \emph{only} relevant decomposition of every standard order; Remark~\ref{rem:converse} gives a counterexample. Accordingly, the witness strategy of \cite[\S8.0.10]{MOS} seeks, for a multiset that is not of Speh type, a single standard order admitting no relevant decomposition. Theorem~\ref{thm:counter} shows that such an order cannot always be chosen with non-increasing endpoints.

\emph{(3) The representation-theoretic converse.}
The implication established in \cite[\S9]{MOS} runs from a symplectic period of the standard module $\lambda(\mathfrak{M})$ to distinguishedness of the unlabelling of $\mathfrak{M}$. Distinguishedness is a necessary combinatorial condition; the arguments here do not establish that it is sufficient for $\lambda(\mathfrak{M})$ to admit a symplectic period. This converse should also be distinguished from the corresponding question for the irreducible quotient of $\lambda(\mathfrak{M})$. Corollary~\ref{cor:rep} uses only the established necessary implication.

\subsection*{Acknowledgements}
The author thanks Mohan Verma (VIPS-TC, New Delhi) for many helpful discussions and for his constant encouragement during the preparation of this article. The author acknowledges the Indian Institute of Technology Bombay for supporting this research through the Institute Postdoctoral Fellowship.

\end{document}